\documentclass[reqno, 12pt, reqno]{amsart}

\usepackage{amsfonts, amsthm, amsmath, amssymb}

\usepackage[usenames,dvipsnames]{color}
\usepackage{hyperref,tikz}

  \usetikzlibrary{positioning}

\hypersetup{colorlinks=false}

\usepackage{yhmath}

\usepackage[margin=3.5cm]{geometry}

\RequirePackage{mathrsfs} \let\mathcal\mathscr
\numberwithin{equation}{section}

\renewcommand{\phi}{\varphi}
\renewcommand{\rho}{\varrho}

\newcommand{\0}{\mathbf{0}}
\newcommand{\PP}{\mathbb{P}}

\newcommand{\A}{\mathbf{A}}
\newcommand{\FF}{\mathbb{F}}
\newcommand{\ZZ}{\mathbb{Z}}
\newcommand{\ZZp}{\mathbb{Z}_{\mathrm{prim}}}
\newcommand{\NN}{\mathbb{N}}
\newcommand{\QQ}{\mathbb{Q}}
\newcommand{\RR}{\mathbb{R}}

\newcommand{\Gal}{{\rm Gal}}

\renewcommand{\leq}{\leqslant}

\renewcommand{\geq}{\geqslant}

\renewcommand{\bar}{\overline}

\newcommand{\x}{\mathbf{x}}
\newcommand{\y}{\mathbf{y}}

\renewcommand{\v}{\mathbf{v}}
\renewcommand{\u}{\mathbf{u}}

\newcommand{\fo}{\mathfrak{o}}

\newcommand{\fp}{\mathfrak{p}}

\newcommand{\ve}{\varepsilon}

\DeclareMathOperator{\rank}{rank}

\DeclareMathOperator{\Br}{Br}
\DeclareMathOperator{\frob}{Frob}
\DeclareMathOperator{\dens}{dens}
\DeclareMathOperator{\Spec}{Spec}
\DeclareMathOperator{\Fix}{Fix}
\DeclareMathOperator{\Stab}{Stab}

\DeclareMathOperator{\FEt}{\mathbf{FEt}}
\DeclareMathOperator{\Sets}{\mathbf{Sets}}
\DeclareMathOperator{\rad}{rad}
\DeclareMathOperator{\supp}{supp}

\DeclareMathOperator{\Bl}{Bl}

\renewcommand{\t}{\mathbf{t}}

\newcommand{\1}{\mathbf{1}}

\newtheorem{thm}{Theorem}[section]
\newtheorem{cor}[thm]{Corollary}
\newtheorem{prop}[thm]{Proposition}
\newtheorem{lemma}[thm]{Lemma}
\newtheorem{conjecture}[thm]{Conjecture}

\theoremstyle{definition}
\newtheorem{defin}[thm]{Definition}
\newtheorem{rem}[thm]{Remark}
\newtheorem{example}[thm]{Example}

\numberwithin{equation}{section}

\begin{document}
\date{\today}

\title{Paucity of rational points on fibrations with multiple fibres}

\author{Tim Browning}

\address{IST Austria\\
Am Campus 1\\
3400 Klosterneuburg\\
Austria}
\email{tdb@ist.ac.at}

\author{Julian Lyczak}
\address{
Department of Mathematical Sciences\\ University of Bath\\ Claverton Do\-wn\\ Bath BA2
7AY\\ UK}
\email{jl4212@bath.ac.uk}

\author{Arne Smeets}
\address{Department of Mathematics\\
KU Leuven\\
Celestijnenlaan 200B\\
B-3001 Leuven\\
Belgium}

\email{arne.smeets@kuleuven.be}


\begin{abstract}
Given a family of varieties over the projective line, we  study the density of fibres that are everywhere locally soluble in the case that components of higher multiplicity are allowed.  
We use log geometry to formulate a new sparsity criterion for the existence of everywhere locally soluble fibres and formulate new conjectures that generalise previous work of Loughran--Smeets. These conjectures involve geometric invariants of the associated multiplicity orbifolds on the base of the fibration in the spirit of Campana. We give evidence for the conjectures using Chebotarev's theorem and sieve methods.
\end{abstract}

\date{\today}
\maketitle

\thispagestyle{empty}
\setcounter{tocdepth}{1}
\tableofcontents

\section{Introduction}\label{sec:intro}

Let $X$ be a smooth, proper, geometrically irreducible variety over  $\QQ$, which is equipped with a dominant morphism $\pi:X\to \PP^1$  with geometrically integral generic fibre.  
We shall refer to such fibrations as {\em standard}.
The focus  of this article is on situations  where multiple fibres are present. Work of Colliot-Th\'el\`ene, Skorobogatov and Swinnerton-Dyer \cite{paucity} shows that the set $X(\QQ)$ of $\QQ$-rational points on $X$ is not Zariski dense  when there are at least $5$ geometric double fibres.  
Our  goal is to put this kind of result on a quantitative footing by analysing  the simpler question of
solubility over 
the ring of ad\`eles
 $\mathbf{A}_\QQ$.
Let 
\begin{equation}\label{eq:count}
N_{\text{loc}}(\pi,H,B)=\#\left\{
x\in \PP^1(\QQ)\cap \pi(X(\mathbf{A}_\QQ)): H(x)\leq B
\right\},
\end{equation}
where $H$ is a height function on $\PP^1(\QQ)$.
In general, we will need to allow the height $H$ to be any adelic height on a line bundle $\mathcal O(d)$. However, most of the time we shall use an $\mathcal O(1)$-height.
In this case we will simply write $N_{\text{loc}}(\pi,H,B)=N_{\text{loc}}(\pi,B)$.
Usually we will take the naive height $H(x) = \max\{|x_0|,|x_1|\}$, if $x\in \PP^1(\QQ)$ is represented by a vector $\x=(x_0,x_1)\in \ZZp^2$, in which case it is easy to prove that $\#\{x\in \PP^1(\QQ): H(x)\leq B\}\sim  
\frac{2}{\zeta(2)}
B^{2}$, 
as $B\to \infty$. 

Loughran and Smeets \cite{LS} 
 have shown that 
\begin{equation}\label{eq:upper}
N_{\text{loc}}(\pi,B)\ll \frac{B^2}{(\log B)^{\Delta(\pi)}},
\end{equation}
for a certain exponent $\Delta(\pi)\geq 0$ that is defined in terms of the  data of the fibration.
(Here, as throughout our work, all implied constants are allowed to depend on the fibration $\pi$.)
Roughly speaking,  the size of  $\Delta(\pi) $ 
is determined by the number of non-split fibres of   $\pi$, thereby lending credence to a philosophy put forward by Serre \cite{Ser90} and further developed  by Loughran \cite{L}.
In \cite[Conj.~1.6]{LS} it is conjectured that the upper bound 
\eqref{eq:upper} is sharp provided that the fibre of $\pi$ over every closed  point of $\PP^1$ has an irreducible component of multiplicity one.  (In fact, the work in \cite{LS} works over arbitrary number fields $k$ and concerns fibrations $X\to \PP^n$ over projective space of arbitrary dimension, but we shall restrict to $k=\QQ$ and $n=1$ in our work.)
Our goal is to  explore what happens to $N_{\text{loc}}(\pi,B)$
when the assumption about components of multiplicity one is violated. 

There are relatively few examples in the number theory literature that feature standard fibrations with multiple fibres.
When the generic fibre of $\pi$ is  rationally connected, it follows from work of 
Graber, Harris and Starr \cite{GHS} that every fibre contains a geometrically integral component of multiplicity one. 
In particular, when $\dim X=2$,  we must look to fibrations over $\PP^1$ into curves of positive genus to find examples with  multiple fibres. 
Let $c,d,f\in \QQ[t]$  be non-zero polynomials
such that $f$ is  square-free of even degree and such that $f$ and $c-d$ are coprime. 
Let $\pi \colon X\to \PP^1$ be a smooth, proper model of the affine variety cut out by the pair of equations
\begin{equation}\label{eq:colliot}
x^2-c(t)=f(t)y^2, \quad x^2-d(t)=f(t)z^2.
\end{equation}
Then it follows from 
\cite[Prop.~4.1]{paucity}  that all the fibres of $\pi$ over the zeros of $f$ are double fibres, and  that the generic fibre is a geometrically integral curve whose projective model is isomorphic to a curve of genus one. 
When $\deg(f)\geq 6$, as pointed out by Loughran and Matthiesen \cite[Thm.~1.4]{LM}, 
the argument of  \cite[Cor.~2.2]{paucity} implies 
that $N_{\text{loc}}(\pi,B)=O(1)$.
Further examples involving genus $2$ fibrations over $\PP^1$ have been worked out by 
Stoppino \cite{stop}.

In the spirit of Campana
\cite{campana},
our approach to this problem comes from relating the arithmetic of $\pi \colon X\to\PP^1$ to the
arithmetic of the {\em orbifold base} $(\PP^1,\partial_\pi)$, for a certain   $\QQ$-divisor $\partial_\pi$, in the sense of Definition \ref{def:rose}.
 For each closed point  $D\in (\PP^1)^{(1)}$, we let $m_D\geq 1$ denote the minimum multiplicity of the irreducible components of $\pi^{-1}(D)$. Then we may define
\begin{equation}\label{eq:def_partial}
\partial_\pi=\sum_{D\in (\PP^1)^{(1)}} \left(1-\frac{1}{m_D}\right) [D].
\end{equation}
With this notation, we make the following conjecture.

\begin{conjecture}\label{con1}
Let  $\pi:X\to \PP^1$ be a standard fibration such that the $\QQ$-divisor  $-(K_{\PP^1}+\partial_\pi)$ is ample.  Then 
$$
N_{\text{loc}}(\pi,B)= O_{\ve}\left(B^{2-\deg \partial_\pi+\ve }\right),
$$
for any $\ve>0$.
\end{conjecture}

Note that $-\deg(K_{\PP^1}+\partial_\pi)=2-\deg \partial_\pi.
$ Hence
$-(K_{\PP^1}+\partial_\pi)$ is ample if and only if $\deg \partial_\pi<2$.
The main feature of Conjecture~\ref{con1} is 
that we expect
$N_\text{loc}(\pi,B)$ to be much smaller in the presence of  multiple fibres.
Our remaining results give evidence towards this, as well as a proposal about the replacement of $B^\ve$ by an explicit  {\em non-positive} power of $\log B$.

\subsection{Upper bounds}

For each closed point $D\in (\PP^1)^{(1)}$,  let
$S_D$  be the set of geometrically irreducible components of $\pi^{-1}(D)$ of multiplicity $m_D$
and let $\kappa(D)$ be the residue field. 
 For any number field  $N/\QQ$, we write
\begin{equation}\label{eq:dD}
\delta_{D,N}(\pi)=
\frac{\#\{
\sigma\in \Gamma_{D,N}:   \text{$\sigma$ acts with a fixed point on 
$S_{D}$}\}}
{\#\Gamma_{D,N}},
\end{equation}
where 
$\Gamma_{D,N}$  is a finite group through which the action of $\Gal(\overline{N}/N)$ on $S_{D}$ factors.
(We take 
$\delta_{D,N}(\pi)=0$  when no such components exist.)
Note that 
\begin{equation}\label{eq:01}
0\leq \delta_{D,N}(\pi) \leq 1.
\end{equation}
Moreover, we shall write 
$\delta_{D}(\pi)=
\delta_{D,\kappa(D)}(\pi)$. 
When 
$\pi^{-1}(D)$ has components of multiplicity  one, this  agrees with the definition given by Loughran and Smeets \cite[Eq.~(1.4)]{LS}.
A natural analogue of the exponent appearing in \cite[Thm.~1.2]{LS}
is then 
\begin{equation}\label{eq:Delta}
\Delta(\pi)=\sum_{D\in (\PP^1)^{(1)}} \left(1-\delta_D(\pi)\right),
\end{equation}
which agrees with the exponent 
appearing in  \eqref{eq:upper} whenever  $\pi^{-1}(D)$ 
contains a multiplicity one component for every $D\in (\PP^1)^{(1)}$.

The following upper bound treats the case of one multiple fibre above a degree $1$ point of $\PP^1$, and is consistent with Conjecture~\ref{con1}.

\begin{thm}\label{t:1}
Let  $\pi:X\to \PP^1$ be a standard fibration with 
a unique  multiple fibre at $0$. 
Then 
$$
N_{\text{loc}}(\pi,B)\ll \frac{B^{2-\deg \partial_\pi} }{(\log B)^{\Delta(\pi)}},
$$
where
$
\Delta(\pi)$ is given by \eqref{eq:Delta}.
\end{thm}

It is tempting to suppose that the same estimate continues to hold when there is more than one closed point of $\PP^1$ above which multiple fibres exist.
However, in Theorem \ref{t:3}, we shall illustrate that a smaller exponent than $\Delta(\pi)$ is sometimes necessary.

Let $\pi:X\to \PP^1$ be a standard fibration and 
let
  $D\in (\PP^1)^{(1)}$,
which we suppose is defined by an irreducible binary form $g\in \QQ[x,y]$. 
Assume first that $g(1,0)\neq 0$. Then the  residue field is   $\kappa(D)=\QQ[x]/(g(x,1))$.
Moreover, 
for any $d\in \NN$ and any $v\in \QQ$, let 
$
N_{D,d,v}=\QQ[x]/(h(x)),
$
where 
$$
h(x)=g(x^d,v).
$$
For typical $v$ this forms a number field of degree $\deg(g)+d$, but in general forms  an \'etale algebra, since $h$ is not necessarily irreducible, with a factorisation
\begin{equation}\label{eq:wood}
N_{D,d,v}=N_{D,d,v}^{(1)}\times\dots\times N_{D,d,v}^{(s_D)}.
\end{equation}
It still remains to deal with the case $g(1,0)=0$. But then $D=\infty$ and it readily follows that 
$\kappa(\infty)=\QQ[y]/(g(1,y))=\QQ$ and 
$N_{\infty,d,v}=\QQ[y]/(g(1,vy^d))=\QQ$,
for any $v\in \QQ$.
We may now define
\begin{equation}\label{eq:theta-v}
\Theta_v(\pi)=
\sum_{D\in \left(\PP^1\right)^{(1)}} \sum_{k=1}^{s_D}\left(1-\delta_{D,N_{D,d,v}^{(k)}}(\pi)\right),
\end{equation}
in the notation of \eqref{eq:dD}.
Our main upper bound  is as follows.

\begin{thm}\label{t:2}
Let  $\pi:X\to \PP^1$ be a standard fibration with 
 multiple fibres at $0$ and $\infty$, and nowhere else. 
Let   $d=\gcd(m_0,m_\infty)$. Then 
$$
N_{\text{loc}}(\pi,B)\ll \frac{B^{2-\deg \partial_\pi} }{(\log B)^{\min_{v\in 
\QQ^\times/\QQ^{\times,d}} \Theta_v(\pi)}}.
$$
\end{thm}

It will be convenient to  put 
\begin{equation}\label{eq:put}
\Theta(\pi)=
\min_{v\in  \QQ^\times/\QQ^{\times,d}} \Theta_v(\pi).
\end{equation}
Let us first note that
 $\Theta(\pi)\geq 0$, by  \eqref{eq:01}.  
Secondly, $\Delta(\pi)$ and $\Theta(\pi)$ can be different; in Theorem \ref{t:3} we will see an example with $\Theta(\pi)=0$, but $\Delta(\pi)=1$. However, we will see that
\begin{equation}\label{eq:check}
	\Theta(\pi)=\Delta(\pi), \quad \text{if $\gcd(m_0,m_\infty)=1$}.
\end{equation}
The following result shows that there are only finitely many
values that $\Theta_v(\pi)$ can take.

 \begin{thm}\label{t:remark}
Let $\pi:X\to \PP^1$ be a standard fibration and let  $D\in (\PP^1)^{(1)}$. 
Let $E$ be the field of definition of the elements of $S_D$ and let 
$N/\QQ$ be a number field. 
Then 
$
\delta_{D,N}(\pi)=
\delta_{D,N\cap E^{\text{normal}}}(\pi),
$
where $E^{\text{normal}}$ is the 
normal closure  of $E$. 
 \end{thm}

As we have seen, our understanding of 
$N_{\text{loc}}(\pi,B)$ is inexorably linked to the arithmetic of the
orbifold base $(\PP^1,\partial_\pi)$. 
The study of rational points on orbifolds  is the focus of work by
Pieropan, Smeets, Tanimoto and V\'arilly-Alvarado \cite{c-manin}, which  offers a far-reaching conjectural asymptotic formula for any orbifold $(Y,\partial)$ with $\QQ$-ample divisor 
$-(K_{Y}+\partial)$.
Pieropan and Schindler \cite{c-manin'}  have  verified many cases of the conjecture when $Y$ is a split toric variety 
over $\QQ$. Their work covers the orbifolds that arise in the proof of Theorem \ref{t:1} and \ref{t:2} and would yield the upper bound 
$
N_{\text{loc}}(\pi,B)=O(B^{2-\deg \partial_\pi}).
$ 
In order to achieve the desired non-positive powers of $\log B$, we need to 
incorporate extra Chebotarev type conditions that arise when counting locally soluble fibres.

The proofs of Theorems \ref{t:1} and \ref{t:2} are based on the large sieve and will be carried out in Section \ref{s:large}. A crucial ingredient will be a
{\em sparsity criterion}, which gives explicit control over which fibres are everywhere locallly soluble.  This criterion will be proved in  Section \ref{s:sparsity} using log geometry, and may be of independent interest.

Extending Theorem \ref{t:2} to three multiple fibres  represents a formidable challenge. 
The easiest such case corresponds to the $\QQ$-divisor
$$
\partial_\pi=\frac{1}{2}[0]+\frac{1}{2}[1]+\frac{1}{2}[\infty].
$$ 
Conjecture \ref{con1} would predict that 
$N_{\text{loc}}(\pi,B)=O_{\ve}(B^{1/2+\ve})$, for any $\ve>0$. However,  the best upper bound we have 
is due to Browning and 
Van Valckenborgh
\cite{bvv}, which  only  yields the exponent $3/5+\ve$.

\subsection{A new conjecture}

We are now ready to reveal a new conjecture for the density of locally soluble fibres for standard fibrations, in which multiple fibres are allowed.
Let $\pi \colon X \to \mathbb P^1$ be a standard fibration, and let 
$\theta \colon \mathbb P^1 \to (\mathbb P^1,\partial_\pi)$ be a finite \'etale orbifiold morphism, 
as defined in Definition \ref{def:orb_mor}.

We assume that 
$(\mathbb P^1,\partial_\pi)$ does not admit a finite \'etale orbifold morphism which factors through $\theta$, and $\theta$ is a $G$-torsor under 
a
finite \'etale group scheme $G$.  Let  $\theta_{v} \colon \mathbb P^1 \to \mathbb P^1$ denote the  twist of $\theta$ by any $v \in \textup{H}^1(\Gal(\bar\QQ/\QQ),G)$.
Finally, let  $\pi_{v} \colon X_{v} \to \mathbb P^1$ denote the  normalisation of the pullback of $\pi$ along $\theta_{v}$.

\begin{conjecture}\label{con2}
Let  $\pi:X\to \PP^1$ be a standard fibration such that the $\QQ$-divisor
$-(K_{\PP^1}+\partial_\pi)$ is ample and 
$X(\A_\QQ)\neq \emptyset$.
Then there exists a constant $c_{\pi}>0$ such that 
$$
N_{\text{loc}}(\pi,B)\sim c_{\pi} \frac{B^{2-\deg \partial_\pi}}{(\log B)^{
\min_{v\in \textup{H}^1(\Gal(\bar\QQ/\QQ),G)}\Delta(\pi_v)}}
$$
where
$
\Delta(\pi_v)$ is given by \eqref{eq:Delta}.
\end{conjecture}

Note that it follows from Theorem \ref{t:remark} that 
$\Delta(\pi_v)$ takes only finitely many values.
In the special case that the orbifold base is simply connected as an orbifold, 
which in   the setting of Theorem \ref{t:2} covers the case $\gcd(m_0,m_\infty)=1$, we will have
$\Theta(\pi)=\Delta(\pi)$. Thus Conjecture \ref{con2} implies that 
$$
N_{\text{loc}}(\pi,B)\sim c_{\pi} \frac{B^{2-\deg \partial_\pi}}{(\log B)^{\Delta(\pi)}},
$$
in this case, 
which is consistent with the upper bound in Theorem \ref{t:1}.
In Corollary~\ref{c:goat} we  shall take $G=\mu_d$ and prove that
$
\Theta_v(\pi)=\Delta(\pi_v)
$
in \eqref{eq:theta-v}. Hence the upper bound in 
Theorem~\ref{t:3} is also  consistent with Conjecture \ref{con2}.
In Section~\ref{s:examples} we shall provide further evidence for the conjecture, by establishing a range of estimates for 
the variant 
$N_{\text{loc},S}(\pi,B)$
of  $N_{\text{loc}}(\pi,B)$, in which local solubility is only required 
away from a finite set $S$ of primes.
In  Theorem \ref{t:ONE}, for example, we establish a precise lower bound
for $N_{\text{loc},S}(\pi,B)$ in the case that  $\pi:X\to \PP^1$ is a standard fibration for which the only non-split fibres lie over $0$ and $\infty$.

One further source of examples
that can be used to 
 illustrate our conjectures 
 is the class of {\em Halphen surfaces}. These were introduced by 
  Halphen \cite{Halphen} in 1882 and correspond to 
standard fibrations  admitting a unique multiple fibre.
In Theorems~\ref{t:TWO}--\ref{t:FIVE} we provide several estimates for
$N_{\text{loc},S}(\pi,B)$ that are consistent with Conjecture \ref{con2}, 
for appropriate surfaces of Halphen type. 
In the proof of  Theorem~\ref{t:FIVE} we are led to  a
concrete 
 problem in 
analytic number theory that was solved by 
Friedlander and Iwaniec \cite[Thm.~11.31]{FI}. Indeed, we need matching upper and 
lower bounds for the number of positive integers $a,b$ 
satisfying
$a^6+b^2\leq x$, as $x\to \infty$, 
such that the only prime divisors of $a^6+b^2$ are those that split in a given cubic Galois extension $K/\QQ$. It would be useful to have a similar result for non-Galois extensions, but this appears to be difficult.

\begin{rem}
Returning to the example \eqref{eq:colliot}, we see that the associated $\QQ$-divisor $\partial_\pi$ has degree $\frac{1}{2}\deg(f)$.
Since $f$ is assumed to have even degree, it follows  that  $-(K_{\PP^1}+\partial_\pi)$ is ample only when  $\deg(f)=2$.
When $f$ is a quadratic polynomial, 
Conjecture~\ref{con1} implies that 
$N_\text{loc}(\pi,B)=O_\ve(B^{1+\ve})$ for any $\ve>0$.  The orbifold base 
 $(\mathbb P^1,\partial_\pi)$ 
 admits  $\mu_2$-covers and
it is possible to apply Conjecture \ref{con2} to predict an explicit power of $\log B$.
The outcome will depend on the Galois action on the geometric components of the fibres.
\end{rem}

\subsection{Further questions}
We expect similar conjectures to hold when looking at fibrations $\pi:X\to Y$ over other bases for which  $-(K_Y+\partial_\pi)$ is $\QQ$-ample. However, when $\dim(Y)>1$ the sparsity criterion we work out in Section \ref{s:sparsity} will be significantly more complicated.
Moreover, care also needs to be taken around the effect of thin subsets of $Y(\QQ)$ on the counting problem. A counter-example to the most naive expectation has recently been provided 
\cite{bls} in the case that $Y$ is a split quadric in $\PP^3$.

In a different direction, when $Y=\PP^1$, we can 
 extend the definition \eqref{eq:count} by defining 
$
N_{\text{loc}}(\pi,B;Z)$ to be the number of 
$x\in (\PP^1(\QQ)\setminus Z)\cap \pi(X(\mathbf{A}_\QQ))$ for which $H(x)\leq B$, 
for any {\em thin subset} $Z\subset \PP^1(\QQ)$. 
It is then very natural to ask whether or not we should 
expect a bound of the shape 
\begin{align*}
 N_{\text{loc}}(\pi,B;Z) 
 \ll  \frac{B^{\frac{1}{m_0}+\frac{1}{m_\infty}}}{(\log B)^{\Delta(\pi)}},
\end{align*}
where $\Delta(\pi)$ is given by \eqref{eq:Delta}, if we have the freedom to remove any  thin set $Z$.

\subsection{Summary of the paper}

The main sparsity criterion for locally soluble fibres is Theorem \ref{thm:lifting}. It is 
 proved using log geometry in Section \ref{s:sparsity} and leads to Chebotarev type conditions about the splitting behaviour of primes. 
In Section~\ref{s:grouptheoretic} we shall collect together some basic group-theoretic results that allow us to interpret the output from Chebotarev's theorem. 
Section \ref{sec:proofofmaincount} uses recent work of 
Arango-Pi\~neros, Keliher and Keyes 
\cite{arxiv} to count pairs of power-full integers which lie in the multiplicative span of Frobenian sets of primes. 
In Section \ref{sec:orb}
we shall introduce the necessary background on orbifolds that is required to interpret the exponent of $\log B$ in Conjecture \ref{con2}. 
 Section \ref{s:large} contains the proof of Theorems~\ref{t:1} and \ref{t:2} and is based on an application of the large sieve. Finally, Section \ref{s:examples} builds on the work in Section \ref{sec:proofofmaincount}
and 
contains a range of estimates for the modified counting function $N_{\text{loc},S}(\pi,B)$ in specific examples.

\subsection*{Acknowledgements}
We are very grateful to Tim Santens for useful conversations.
While working on this paper
the first    author was supported  by 
 FWF grant P 36278.

\section{Group-theoretic results}\label{s:grouptheoretic}

We will need some preliminary results on the density of primes with a prescribed splitting behaviour. Using Chebotarev's theorem we will be able to  translate it into statements about  groups and group actions. We begin by proving some results in elementary group theory.

\subsection{Group theory lemmas}

Let $G$ be a finite group and let  $H \subseteq G$ be a  subgroup. For  an element $g \in G$ we will write $\Fix_g(G/H)$ for the set of fixed points of $g$ under the natural action of $G$ on $G/H$.

\begin{lemma}\label{l:fixpoints}
Let $C \subseteq G$ be a conjugacy class. Then we have 
\[
\sum_{g \in C} \# \Fix_g(G/H) = \frac{\#G}{\#H} \#(C \cap H).
\]
\end{lemma}

\begin{proof}
First note that for conjugate elements $g,y \in C$ there is an element $u \in G$ such that $u^{-1}yu=g$. Hence
\[
\{x\in G \colon x^{-1}gx =y\} = \{x\in G \colon (ux)^{-1}y (ux) =y\} = u^{-1}\Stab_y,
\]
whose cardinality is $\#G/\#C$ by the orbit--stabiliser theorem, since $C$ is the orbit of $y$ under conjugation. We now see that
\begin{align*}
\sum_{g \in C} \# \Fix_g(G/H) & = \#\{(g,xH) \in C \times G/H \colon gxH=xH\}\\
 & = \#\{(g,xH) \in C \times G/H \colon x^{-1}gx \in H\}.
 \end{align*}
 Hence
\begin{align*}
\sum_{g \in C} \# \Fix_g(G/H) 
 & = \frac1{\#H} \#\{(g,x) \in C \times G \colon x^{-1}gx \in H \cap C\}\\
 & = \frac1{\#H} \#\{(g,x,y) \in C \times G \times (H \cap C) \colon x^{-1}gx=y\}\\
 & = \frac1{\#H}\#C\cdot \frac{\#G}{\#C} \cdot \#(C\cap H),
\end{align*}
which proves the lemma.
\end{proof}

\begin{lemma}\label{l:productset} Let $S$ and $T$ be subgroups of $G$. Then
\[
\#S\#T=\#(S\cap T)\#(ST).
\]
\end{lemma}

\begin{proof}
Consider the action $S\times T$ on $G$ by $(s,t)g=sgt^{-1}$. The stabiliser of $e_G$ equals the image of diagonal map $S \cap T \hookrightarrow S \times T$ and the set $ST$ is the orbit of $e_G$. The result now follows from the orbit--stabiliser formula.
\end{proof}

\subsection{Density of primes}\label{s:density}

Let $F/\QQ$ be a number field with ring of integers $\mathcal O_F$. Define $\mathcal P_{F,m}$ to be the set of rational primes $p$ unramified in $F$ which are divisible by exactly $m$ primes $\mathfrak p_i \subseteq \mathcal O_F$ of degree $1$. Let
\[
\mathcal P_F = \bigcup_{m\geq 1} \mathcal P_{F,m}.
\]
We  define
$$
\delta(E,K) = 1 - \sum_{m=1}^d m \dens\left(\mathcal P_{K,m} \cap \mathcal P_E\right),
$$
for any number fields $K,E \subseteq \bar{\mathbb Q}$ with $d= [K \colon \mathbb Q]$.  
The main result of this section is the following result. 

\begin{thm}\label{t:densities}
Let $K,E \subseteq \bar{\mathbb Q}$ be two number fields with $d= [K \colon \mathbb Q]$. Define
\[
\delta(E,K) = 1 - \sum_{m=1}^d m \dens\left(\mathcal P_{K,m} \cap \mathcal P_E\right).
\]
Let $L \subseteq \bar{\mathbb Q}$ be a Galois extension of $\mathbb Q$ which contains both $K$ and $E$. Then
\[
\delta(E,K) = 1 - \frac{\#\{\sigma \in \Gal(L/K) \colon \sigma \text{ fixes a conjugate of } E\}}{\#\Gal(L/K)}.
\]
\end{thm}

The quantity $\delta(E,K)$  generalises a quantity that is implicit in the work of Loughran--Smeets 
\cite[Eq.~(1.4)]{LS}. Let $\pi:X\to \PP^1$ be a standard fibration and let 
$D$ be a closed point of $\PP^1$ with residue field $\kappa(D)$. 
Let $I_D(\pi)$  be the set of geometrically irreducible 
components of $\pi^{-1}(D)$
of multiplicity one and let 
 $E$ be the minimal extension of $\kappa(D)$ over which the components 
 of $I_D(\pi)$ are defined. Then 
 $$
\delta_D(\pi)=1-
\delta(E,\kappa(D))
$$
in \cite[Eq.~(1.4)]{LS}. Moreover, 
if we take $S_D$ to be the set of geometrically irreducible 
components of $\pi^{-1}(D)$
of multiplicity $m_D$ and we let 
 $E$ be the
 field of definition of the elements of $S_D$, then we also have
\begin{equation}\label{eq:new-delta}
\delta_{D,N}(\pi)=1-
\delta(E,N)
\end{equation}
in \eqref{eq:dD}, for any number field $N/\QQ$.

\begin{proof}[Proof of Theorem \ref{t:densities}]
Write $G= \Gal(L/\mathbb Q)$ and let $K$ and $E$ be the fixed fields of the subgroups $H_1, H_2 \subseteq G$. Then we have
\[
\mathcal P_{K,m} =\{\text{primes } p \in \mathbb Z \text{ unramified in $L$ for which } \#\Fix_{\frob_p}(G/H_1) = m\}
\]
and
\[
\mathcal P_{E} =\{\text{primes } p \in \mathbb Z \text{ unramified in $L$ for which }  \#\Fix_{\frob_p}(G/H_2) \geq 1\}.
\]
Note that
\[
C_m = \{g \in G \colon \#\Fix_g(G/H_1) = m \text{ and } \#\Fix_g(G/H_2) \geq 1\}
\]
is closed under conjugation, since conjugate elements have the same number of fixed points. By Chebotarev's theorem, 
in the form presented by Serre \cite[Thm.~3.4]{serre}, for example,
we therefore obtain
\[
\dens(\mathcal P_{K,m} \cap \mathcal P_{E}) = \frac{\#C_m}{\#G}.
\]

Let $T = \bigcup_{t \in G} tH_2 t^{-1}$, which we
note is closed under conjugation.
Since $g \in G$ has at least a fixed point on $G/H_2$ if and only if $g \in T$,  we arrive at
\begin{align*}
\sum_{m=1}^d m \dens\left(\mathcal P_{K,m} \cap \mathcal P_E\right) & = \frac1{\#G} \sum_{m=1}^d m \#C_m\\
& = \frac1{\#G}\sum_{g \in T} \# \Fix_g(G/H_1).
\end{align*}
We may now conclude from Lemma~\ref{l:fixpoints} that
\begin{equation}\label{eq:milk}
\sum_{m=1}^d m \dens\left(\mathcal P_{K,m} \cap \mathcal P_E\right) = \frac{\#(T \cap H_1)}{\#H_1}.
\end{equation}
The statement of the theorem follows on noting that 
$H_1 = \Gal(L/K)$ and $T= \{\sigma \in G\colon \sigma \text{ fixes a conjugate of } E\}$.
\end{proof}

Note that we could not have applied the Chebotarev Theorem to $\#(T \cap H_1)$, since $T \cap H_1$ is not necessarily fixed under conjugation in $G$. It is however closed under conjugation in $H_1$.

\subsection{Computation of $\delta$ in specific cases}

Theorem~\ref{t:densities} allows us to compute the density $\delta(E,K)$ in the common Galois closure $L$ of both $K$ and $E$. The following theorem says that this can be reduced to a computation in a Galois closure of $E$.

\begin{prop}\label{t:deltainclosure}
Let $E^\text{normal}$ be the normal closure of $E$ in $\bar{\mathbb Q}$. Then
\[
\delta(E,K)=\delta(E,E^\text{normal}\cap K).
\]
\end{prop}

\begin{proof}
We adopt the notation from the proof of 
Theorem~\ref{t:densities}.
Let $H_2^{\{j\}}$ be the conjugates of $H_2$ indexed by a set $J$. For a set $I \subseteq J$ we write $H_2^I = \bigcap_{i \in I} H^{\{i\}}$. The field $E^\text{normal} \cap K$ corresponds to the subgroup $\langle H_1,H_2^J \rangle \subseteq G$ generated by $H_1$ and $H_2^J$.
(Since $H_2^J$ is normal one can actually show that $\langle H_1,H_2^J \rangle = H_1H_2^J$.)
It follows from  Lemma~\ref{l:productset} that
\[
\frac{\#(S \cap H_1)}{\#H_1} = \frac{\#(S \cap \langle H_1,H_2^J \rangle)}{\#\langle H_1,H_2^J \rangle},
\]
when  $S$ is equal to $H_2^I$ for any $I \subseteq J$. Since both sides are additive in $S$, the statement extends to $S=T= \bigcup_{j \in J} H_2^{\{j\}}$ by the principle of inclusion and exclusion.
\end{proof}

\begin{proof}[Proof of Theorem \ref{t:remark}]
Combine Proposition \ref{t:deltainclosure} with  \eqref{eq:new-delta}.
\end{proof}

Our remaining results summarise some special situations in which we can use 
Theorem \ref{t:densities} and 
Proposition \ref{t:deltainclosure} to  calculate the densities $\delta(E,K)$ easily.

\begin{lemma}
If $E/\QQ$ is Galois, then $\delta(E,K)= 1-\frac{\deg(E \cap K)}{\deg E}$.
\end{lemma}

\begin{proof}
Since $E/\mathbb Q$ is Galois, $E$ is also Galois over $E^\text{normal} \cap K = E \cap K$. Thus we conclude
$
\delta(E,K)= \delta(E,E \cap K) = 1-\frac1{[E\colon E\cap K]}$. 
\end{proof}

\begin{lemma}
If $E \subseteq K$ then $\delta(E,K)=0$.
\end{lemma}

\begin{proof}
Since $K,E$ are the fixed fields of the subgroups $H_1, H_2 \subseteq  \Gal(L/\mathbb Q)$, we have $E \subseteq K$ if and only if $H_2 \supseteq H_1$.
But then  $H_1 \subseteq H_2 \subseteq T
= \bigcup_{t \in G} tH_2 t^{-1}$, whence $\tfrac{\#(T \cap H_1)}{\#H_1} = 1$ in \eqref{eq:milk}.
\end{proof}

\begin{lemma}
If $K/\QQ$ is Galois and  $KE=E^\text{normal}$, then
$
\delta(E,K)=1-\frac{\deg(E\cap K)}{\deg E}.$
\end{lemma}

\begin{proof}
Since  $KE=E^\text{normal}$ and $K/\QQ$ is Galois we have $H_1 \cap H_2^{\{j\}} = H_2^J$ for all $j\in J$. Thus
\begin{align*}
\frac{\#(T\cap H_1)}{\#H_1} = \frac{\#H_2^J}{\#H_1}
 = \frac{\deg K}{\deg E^\text{normal}} = \frac{\deg K}{\deg KE}
\end{align*}
in \eqref{eq:milk}.
Since $K$ is Galois we have $[KE \colon K] = [E \colon E \cap K]$, from which the lemma follows.
\end{proof}

\section{Pairs of integers with Frobenian conditions}\label{sec:proofofmaincount}

We say that a set 
$\mathcal{P}$ of rational primes is {\em Frobenian} if there is a finite Galois extension 
$K/\QQ$ and a union of conjugacy classes $H$ in 
$\Gal(K/\QQ)$ such that $\mathcal{P}$ is 
 equal to the set of primes $p$ 
 that are unramified in $K$ and 
 for which the Frobenius conjugacy class of $p$ in
 $\Gal(K/\QQ)$ lies in $H$. In this section we produce an asymptotic formula for the density of 
 coprime integers $a_0,a_1$ which are  both  power-full  and lie in the multiplicative span of a Frobenian set of primes. 

It will be convenient to introduce the notation 
\begin{equation}\label{eq:CS}
c_S(\alpha)=\prod_{p\in S}\left(1-\frac{1}{p^\alpha}\right),
\end{equation}
for any $\alpha>0$ and any finite set of primes $S$.
We shall prove the following result. 

\begin{prop}\label{prop:maincount}
For $i \in \{0,1\}$ let $m_i\in \NN$ and let $\mathcal P_i$ be a Frobenian set of rational primes of density $\partial_i$. Then, for any finite set of primes $S$, we have 
\begin{align*}
\#&
\left\{  (a_0,a_1)\in \ZZp^2: |a_i|\leq B, ~p\not\in S \Rightarrow \big[ m_i\mid v_p(a_i) \text{ and } (p\mid a_i \Rightarrow p\in \mathcal{P}_i)\big] \right\} \\
&\hspace{4cm}
\sim
c_{m_i,\mathcal P_i,S}\frac{B^{1/m_0+1/m_1}}{(\log B)^{2- \partial_0-\partial_1}},
\end{align*}
as $B\to \infty$,  where
\begin{align*}
c_{m_i,\mathcal P_i,S}=~&
\frac{4m_0^{1-\partial_0}m_1^{1-\partial_1}}
{\Gamma(\partial_0)\Gamma(\partial_1)}
\cdot 
\frac{c_S(\frac{1}{m_0}+\frac{1}{m_1})}{c_S(\frac{1}{m_0})c_S(\frac{1}{m_1})}
\prod_{\substack{p\in \mathcal{P}_0\cap \mathcal{P}_1\\ p\not \in S}}\left(1-\frac{1}{p^2}\right)
\\
&\quad \times 
\prod_{p\in  \mathcal{P}_0\cap S}\left(1-\frac{1}{p}\right)
\prod_{p\in \mathcal{P}_0}\left(1-\frac{1}{p}\right)^{-1+\partial_0}
\prod_{p\not \in \mathcal{P}_0}\left(1-\frac{1}{p}\right)^{\partial_0}\\
&\quad \times 
\prod_{p\in  \mathcal{P}_1\cap S}\left(1-\frac{1}{p}\right)
\prod_{p\in \mathcal{P}_1}\left(1-\frac{1}{p}\right)^{-1+\partial_1}
\prod_{p\not \in \mathcal{P}_1}\left(1-\frac{1}{p}\right)^{\partial_1}.
\end{align*}
\end{prop}

There are only $O(1)$ elements with $a_0a_1=0$ that contribute to 
the counting function. 
Let $M(B)=M(m_i,\mathcal{P}_i,B,S)$ denote the overall  contribution with $a_0a_1\neq 0$.
Hence, on accounting for signs, we have
\begin{align*}	
M(B)
=4
\#\left\{  (a_0,a_1)\in \NN^2: 
\begin{array}{l}
a_0,a_1\leq B,~ \gcd(a_0,a_1)=1\\ 
p\not\in S \Rightarrow \big[ m_i\mid v_p(a_i) \text{ and } (p\mid a_i \Rightarrow p\in \mathcal{P}_i)\big] 
\end{array}
\right\}.
\end{align*}
For $(a_0,a_1)$ appearing in the counting function, 
we may clearly write $$
a_0=b_0u_0^{m_0}\quad \text{ and }\quad
a_1=b_1u_1^{m_1}, 
$$
where  $p\mid b_0b_1 \Rightarrow p\in S$, $\gcd(u_0u_1,\prod_{p\in S}p)=1$,  and 
$p\mid u_i \Rightarrow p\in \mathcal{P}_i$. Moreover, we have $\gcd(b_0,b_1)
=\gcd(u_0,u_1)=1$.
Let $\mathcal{Q}=\mathcal{P}_0\cap\mathcal{P}_1$.

We proceed by introducing the counting functions 
$$
M_i(x)
=\#\left\{
v\leq x: p\mid v \Rightarrow p\in \mathcal{P}_{i,S}\right\},
$$
for $i=0,1$, 
where
$
\mathcal{P}_{i,S}=
\mathcal{P}_i\setminus (S\cap \mathcal{P}_i).
$
On using the M\"obius function to detect the condition 
$\gcd(u_0,u_1)=1$, we may now write
\begin{align*}
M(B)
&=4
\sum_{\substack{b_0,b_1\in \NN\\ \gcd(b_0,b_1)=1\\
p\mid b_0b_1 \Rightarrow p\in S}}
\sum_{\substack{k\in \NN\\
p\mid k\Rightarrow p\in \mathcal{Q}_{S}
}} \mu(k)
M_0\left(k^{-1}(B/b_0)^{1/m_0}\right)
M_1\left(k^{-1}(B/b_1)^{1/m_1}\right),
\end{align*}
where
$
\mathcal{Q}_{S}=
\mathcal{Q}\setminus (S\cap \mathcal{Q}).
$
The treatment of $M_i(x)$ is handled by the following result.

\begin{lemma}\label{lem:M(x)}
Let $i\in \{0,1\}$. Then 
$$
M_i(x)\sim \frac{ \kappa_{i,S}}{\Gamma(\partial_i)} \frac{x}{(\log x)^{1-\partial_i}},
$$
as $x\to \infty$, 
where
\begin{equation}\label{eq:kappaS}
\kappa_{i,S}=
\prod_{p\in  \mathcal{P}_i\cap S}\left(1-\frac{1}{p}\right)
\prod_{p\in \mathcal{P}_i}\left(1-\frac{1}{p}\right)^{-1+\partial_i}
\prod_{p\not \in \mathcal{P}_i}\left(1-\frac{1}{p}\right)^{\partial_i}.
\end{equation}
\end{lemma}

\begin{proof}
Let $i\in \{0,1\}$.
There are several approaches to estimating $M_i(x)$, but the one we shall adopt is via a general result of Wirsing \cite{wirsing} on mean values of  multiplicative arithmetic functions $g:\NN\to [0,1]$. (In fact, this result applies to general
non-negative  multiplicative arithmetic functions under further assumptions on the behaviour of $g$ at prime powers.) Suppose  that 
$$
\sum_{p\leq x} g(p)\log p\sim \tau x,
$$
for some $\tau>0$.
Then it follows that 
$$
\sum_{n\leq x} g(n) \sim \frac{e^{-\gamma \tau}}{\Gamma(\tau)} \frac{x}{\log x} \prod_{p\leq x} \left(1+\frac{g(p)}{p}+\frac{g(p^2)}{p^2}+\dots\right),
$$
where $\gamma$ is Euler's constant.  

In our case we take
$$
g(n)=\begin{cases}
1 &\text{ if $p\mid n \Rightarrow p\in \mathcal{P}_{i,S}$,}\\
0 &\text{ otherwise.}
\end{cases}
$$
 Then, since $ \mathcal{P}_{i}$ is a Frobenian set of primes of density $\partial_i$, it follows 
from the Chebotarev density theorem that 
 \begin{align*}
\sum_{p\leq x} g(p)\log p
&=
\sum_{\substack{p\leq x\\ p\in \mathcal{P}_{i,S}}} \log p \sim \partial_i \log x,
\end{align*}
as $x\to \infty$.
Hence $\tau=\partial_i$ and we obtain
$$
M(x)\sim
 \frac{e^{-\gamma\partial_i}}{\Gamma(\partial_i)} \frac{x}{\log x} \prod_{\substack{p\leq x\\
 p\in \mathcal{P}_{i,S}}}
  \left(1-\frac{1}{p}\right)^{-1},
$$
as $x\to \infty$. It remains to study 
$$
\prod_{\substack{p\leq x\\
 p\in \mathcal{P}_{i,S}}}
  \left(1-\frac{1}{p}\right)^{-1}=
  \prod_{p\in \mathcal{P}_i\cap S}\left(1-\frac{1}{p}\right)
 \prod_{\substack{p\leq x\\ p\in \mathcal{P}_i}}
  \left(1-\frac{1}{p}\right)^{-1}.
$$
However, on appealing to recent work of 
Arango-Pi\~neros, Keliher and Keyes \cite[Thm.~A]{arxiv}, we quickly arrive at the expression
\begin{align*}
\prod_{\substack{p\leq x\\ p\in \mathcal{P}_i}}
  \left(1-\frac{1}{p}\right)^{-1}
  &\sim 
\left(\frac{\log x}{e^{-\gamma_K}}\right)^{\partial_i},
\end{align*}
as $x\to \infty$, 
where
$$
e^{-\gamma_K}=e^{-\gamma} \prod_{p\in \mathcal{P}_i}\left(1-\frac{1}{p}\right)^{\partial_i^{-1}-1}
\prod_{p\not \in \mathcal{P}_i}\left(1-\frac{1}{p}\right)^{-1}.
$$
It now follows that 
$$
\prod_{\substack{p\leq x\\
 p\in \mathcal{P}_{i,S}}}
  \left(1-\frac{1}{p}\right)^{-1}\sim \kappa_{i,S} ( \log x)^{\partial_i} 
  e^{\gamma\partial_i},
  $$
  in the notation of lemma. 
Inserting this into our previous asymptotic formula for $M_i(x)$, we finally arrive at the statement of the lemma.
 \end{proof}

We clearly have 
\begin{align*}
\left(\log \left(k^{-1}(B/b_i)^{1/m_i}\right)\right)^{-(1-\partial_i)}
&=m_i^{1-\partial_i}(\log B)^{-(1-\partial_i)}\left(1+O\left(\frac{\log kb_i}{\log B}\right)\right),
\end{align*}
for $i=0,1$.
Hence, on substituting 
Lemma \ref{lem:M(x)} into our previous expression for $M(B)$, we thereby obtain 
\begin{align*}
M(B)
&=~4
\sum_{\substack{b_0,b_1 \in \NN 
\\ \gcd(b_0,b_1)=1\\
p\mid b_0b_1 \Rightarrow p\in S}}
\sum_{\substack{k\in \NN\\ 
p\mid k\Rightarrow p\in \mathcal{Q}_{S}
}} 
\hspace{-0.2cm}
A_{b_0,b_1,k}(B)
+o\left(\frac{B^{1/m_0+1/m_1}}{(\log B)^{2-\partial_0-\partial_1}}\right),
\end{align*}
with 
\begin{align*}
A_{b_0,b_1,k}(B)
&=\frac{\kappa_{0,S} \kappa_{1,S}}{\Gamma(\partial_0)\Gamma(\partial_1)}\cdot 
\frac{\mu(k) m_0^{1-\partial_0}m_1^{1-\partial_1}\left(k^{-1}(B/b_0)^{1/m_0}\right)\left(k^{-1}(B/b_1)^{1/m_1}\right)}{(\log B)^{2-\partial_0-\partial_1}}\\
&=
\frac{\kappa_{0,S} \kappa_{1,S}}{\Gamma(\partial_0)\Gamma(\partial_1)}\cdot 
m_0^{1-\partial_0}m_1^{1-\partial_1} \cdot 
\frac{B^{1/m_0+1/m_1}}{(\log B)^{2-\partial_0-\partial_1}} \cdot 
\frac{\mu(k)}{k^2} \cdot \frac{1}{b_0^{1/m_0} b_1^{1/m_1}},
\end{align*}
and where $\kappa_{0,S}, \kappa_{1,S}$  are given by 
\eqref{eq:kappaS}

Next, 
on recalling the notation of \eqref{eq:CS}, 
a simple calculation furnishes the identities
$$
\sum_{\substack{b_0,b_1 \in \NN 
\\ \gcd(b_0,b_1)=1\\
p\mid b_0b_1 \Rightarrow p\in S}}
 \frac{1}{b_0^{1/m_0}b_1^{1/m_1}}=
\frac{c_S(\frac{1}{m_0}+\frac{1}{m_1})}{c_S(\frac{1}{m_0})c_S(\frac{1}{m_1})}
$$
and 
$$
\sum_{\substack{k\in \NN\\ 
p\mid k\Rightarrow p\in \mathcal{Q}_{S}
}} \frac{\mu(k)}{k^2}=
\prod_{\substack{p\in \mathcal{P}_0\cap \mathcal{P}_1\\ p\not \in S}}\left(1-\frac{1}{p^2}\right).
$$
Hence, it follows that the asymptotic formula in Proposition 
\ref{prop:maincount} holds with the leading constant
\begin{align*}
c_{m_i,\mathcal P_i,S}&=4
\cdot 
\frac{\kappa_{0,S} \kappa_{1,S}}{\Gamma(\partial_0)\Gamma(\partial_1)}\cdot 
m_0^{1-\partial_0}m_1^{1-\partial_1} \cdot 
\frac{c_S(\frac{1}{m_0}+\frac{1}{m_1})}{c_S(\frac{1}{m_0})c_S(\frac{1}{m_1})} \cdot 
\prod_{\substack{p\in \mathcal{P}_0\cap \mathcal{P}_1\\ p\not \in S}}\left(1-\frac{1}{p^2}\right),
\end{align*}
where $\kappa_{0,S}, \kappa_{1,S}$ are given by 
\eqref{eq:kappaS}. This therefore  completes the proof of 
 Proposition~\ref{prop:maincount}.

\section{Orbifolds and \'etale orbifold morphisms}
\label{sec:orb}

Campana related the study of fibrations $\pi \colon X \to Y$ of varieties over a fixed field $k$ to orbifolds on the base
\cite{campana1}. He studied {\em multiplicity orbifolds}, but since these are the only orbifolds in this paper we will simply call them {\em orbifolds}.
In this section we summarise  the construction of the most important invariant of  orbifolds.

\subsection{Orbifold pairs}

Throughout this section let $k$ be an arbitrary field
of characteristic $0$.

\begin{defin}
An  \textit{orbifold} is a pair $(B, \Delta)$, where $B$ is a normal, proper $k$-scheme and $\Delta$ is a $\mathbb Q$-divisor
\[
\Delta = \sum_{D} \left(1-\frac1{m_D}\right) [D]
\]
for positive integers $m_D$ associated to prime divisors $D$ on $B$.
We call $m_D$ the \textit{multiplicity} of the orbifold over $D$.
\end{defin}

\begin{defin}\label{def:orb_mor}
Let $(B, \Delta)$ be an orbifold on a normal and proper $k$-variety $B$.
A \textit{finite \'etale (orbifold) morphism} is a morphism $\theta \colon C \to B$, with $C$ normal, 
which is
\begin{enumerate}
\item[(i)] finite,
\item[(ii)] \'etale away from $\Delta$,
\item[(iii)] has the property $e(D'/D) \mid m_D$, for any 
prime divisor $D'\mid D$ (meaning any   
 prime divisor $D' \subset C$ above $D \subset B$), 
where $e(D'/D)$ is the ramification index.
\end{enumerate}
\end{defin}

Let us explain the use of the word \'etale.
Consider a finite dominant morphism
$\theta \colon C \to B$ 
 between integral, normal, proper $k$-varieties.  Then we can always 
endow $B$ with an orbifold structure such that $\theta$ becomes a finite \'etale orbifold morphism, by assigning $m_D = \mathrm{lcm}\{e(D'/D) \colon D'\mid D\}$.
If $B$ has an orbifold divisor $\Delta$ under which $\theta$ is a finite \'etale orbifold morphism, then we can endow $C$ with the $\QQ$-divisor
\[
\Delta_C = \sum_{D'} \left(1-\frac1{m_{D'}}\right) [D'], \quad \text{where }m_{D'} = \frac{m_D}{e(D'/D)}.
\]
This is the unique orbifold structure on $C$ such that the orbifold morphism $(C,\Delta_C) \to (B,\Delta)$ is \'etale in codimension 1, in the sense of \cite[Definition~2.21]{campana2011}.
In the latter case,  the Riemann--Hurwitz formula yields
$$
K_{C,\Delta_C} = \theta^* K_{B,\Delta},
$$
where $K_{B,\Delta} = K_B + \Delta$ is the {\em canonical divisor class} on an orbifold $(B,\Delta)$.
(This statement can be proven along similar lines to the proof of Proposition~\ref{prop:thin}(c).)

\begin{prop}\label{prop:multiplicities under normalisation}
Let $C_1,C_2 \to C$ be morphisms of normal $k$-varieties. Let $V = \widetilde{C_1 \times_C C_2}$ be the normalisation of the product $C_1 \times_C C_2$. \begin{center}
\begin{tikzpicture}[auto]
\node (L1) {$C_1 \times_C C_2$};
\node (TL) [above left= .2cm of L1] {$V$};
\node (L2) [below= 0.9cm of L1] {$C_2$};
\node (M1) [right= 0.9cm  of L1] {$C_1$};
\node (M2) at (M1 |- L2) {$C$};
\draw[->] (L1) to node {} (M1);
\draw[->] (L2) to node {} (M2);
\draw[->] (L1) to node {}  (L2);
\draw[->] (M1) to node {} (M2);
\draw[->] (TL) to node {} (L1);
\draw[->] (TL) to node {} (M1);
\draw[->] (TL) to node {} (L2);
\end{tikzpicture}
 \end{center}
Let $D_V \subset V$ be a prime divisor lying above prime divisors $D_i \subset C_i$ and $D \subset C$.
Then
\[
e(D_V/D_1) = \frac{e_2}{\gcd(e_1,e_2)},
\]
where $e_i = e(D_i/D)$ for $i=1,2$.
\end{prop}

\begin{proof}
Replacing the prime divisors with their generic points we can compute the normalisation \'etale locally over $D$. Hence we consider the normalisation of the tensor product of the two homomorphisms $\rho_i \colon k[\![ t ]\!] \to k[\![ t_i ]\!]$ given by $ t \mapsto t_i^{e_i}$. The tensor product is $R=k[\![ t_1,t_2 ]\!]/(t_1^{e_1} - t_2^{e_2})$ generated by the images of the $t_i$. Let us define $d= \gcd(e_1,e_2)$.
We can write $R$ as the product
\[
R = \prod_{\stackrel{p \mid X^d-1}{\text{irr.}}} k[\![ t_1,t_2,\zeta ]\!]/(p(\zeta),t_1^{e_1/d}-\zeta t_2^{e_2/d}).
\]
All factors are principal ideal domains, since polynomials $X^a-\lambda Y^b$ with $\lambda \in k^\times$ and $\gcd(a,b)=1$, are even irreducible over an algebraic closure $\bar{k}$. 
We will compute the integral closure of each component separately. Let us write $\alpha_1 e_1 + \alpha_2 e_2 = d$ for $\alpha_i \in \mathbb Z$. Then $T = t_1^{\alpha_2} t_2^{\alpha_1}$ is integral in each factor, since 
$T^{e_1/d}=t_2 \cdot (t_1^{e_1/d}/t_2^{e_2/d})^{\alpha_2}$ and 
$T^{e_2/d}=t_1 \cdot (t_2^{e_2/d}/t_1^{e_1/d})^{\alpha_1}$.
It follows that 
\[
k[\![ t_1,t_2,\zeta ]\!]/(p(\zeta),t_1^{e_1/d}-\zeta t_2^{e_2/d}) \hookrightarrow k[\![ T, \zeta ]\!]/(p(\zeta))
\]
is the integral closure.
Finally, 
to compute $e(D_V/D_1)$ we look at the image of $t_1$ under the map $$
k[\![ t_1 ]\!] \to k[\![ T, \alpha ]\!]/(p(\alpha)),
$$ 
which has valuation $e_2/d$.
\end{proof}

\begin{rem}\label{rem:cap}
In 
\cite[Definition~11.1]{campana2011}, Campana  defines the orbifold fundamental group $\pi_1(X\vert \Delta)$ for a complex orbifold $(X\vert \Delta)$ and relates it to \textit{covers unramified away from $\Delta$}. Likewise, we can define the  (algebraic) orbifold fundamental group and relate it to  the structure of all  finite \'etale orbifold morphisms over a fixed base $(B,\Delta)$ of dimension $1$. (Note that we could do this in arbitrary dimension, if we allow finite \'etale morphisms to be defined away from a codimension $2$ locus.) 
Consider  the category $\FEt_{(B,\Delta)}$ of all finite \'etale orbifold morphisms to $(B,\Delta)$, where the morphisms are given by $B$-morphisms. Given a point $\bar{x} \in B(\bar{k})\setminus \supp(\Delta)$ we have the fibre functor 
$$
F\colon \FEt_{(B,\Delta)} \to \Sets
$$
given by $C \mapsto C_{\bar{x}}$, and one can show that $(\FEt_{(B,\Delta)}, F)$ is a Galois category. The only non-trivial part is to show that $\FEt_{(B,\Delta)}$ has products, but this follows from Proposition~\ref{prop:multiplicities under normalisation}. In particular, this implies that for any two finite \'etale covers of $(C,\partial)$, there is another cover mapping to both.
We define the {\em (algebraic) orbifold fundamental group} $\pi^\text{orb}_1(B,\Delta)$ to be the automorphism group of the fibre functor $F$.
Many relations between the topological and algebraic fundamental group can be directly translated to fundamental groups of orbifolds. For example, if $k \subset \mathbb C$ then 
\[
\pi^\text{orb}_1(B,\Delta) = \widehat{\pi_1(B(\mathbb C)\vert\Delta)}.
\]
Campana studied the complex orbifold fundamental group in \cite[Sections~11 and 12]{campana2011} and has several results and conjectures about their structure.
\end{rem}

For our application we will need the following definition.

\begin{defin}\label{def:G-torsor}
Let $G/k$ be a finite \'etale group and $(B,\Delta)$ an orbifold. Let $\theta \colon C \to B$ be a finite \'etale orbifold  morphism endowed with a $G$-action on $C$, which is  compatible with $\theta$. We say that $\theta$ is a \textit{$G$-torsor (of orbifolds)} if the restriction of $\theta$ away from the support of $\Delta$ is a $G$-torsor.
\end{defin}

Since we are dealing with curves, it makes sense to talk about torsors. The natural morphism $G \times_B C \to C \times_B C$ is not necessarily an isomorphism over $B$, but it is so over $B\setminus \Delta$ by definition. Since $G \times_B C$ is a smooth curve over $k$, this morphism factors through the normalisation $G \times_B C \to \widetilde{C \times_B C} \to C \times_B C$. Now $G \times_B C \to \widetilde{C \times_B C}$ is morphism between normal curves, which is an isomorphism on a dense open subset. Note that this agrees with the observation that $\widetilde{C \times_B C} \to C$ is unramified by Proposition~\ref{prop:multiplicities under normalisation}; $\widetilde{C \times_B C}$ is just a union of copies of $C$.

Categorically, the product of two normal covers of $B$ is the normalisation of the usual fibre product, which means that a  $G$-torsor of orbifolds is indeed a torsor.

\subsection{Orbifold base of a fibration}
As we saw in Section~\ref{sec:intro}, we  can associate a natural orbifold to any fibration. 
In this section we discuss this further, before passing to our reasoning behind  Conjecture \ref{con2}.

\begin{defin}\label{def:rose}
Consider a fibration $\pi\colon X \to Y$, which we assume is  a morphism between integral, normal, proper $k$-schemes such that the generic fibre is geometrically irreducible.
For a prime divisor $D \subset Y$ with generic point $\eta_D$ we define $m_D$ as the minimum multiplicity of the components of $X_{\eta_D}$ as a divisor on $X$. The \textit{orbifold base} of $\pi$ is $(Y,\partial_\pi)$ where
\[
\partial_\pi = \sum_D \left( 1- \frac1{m_D}\right) [D].
\]
\end{defin}

Possibly up to thin sets, we expect the geometry of the base orbifold $(Y,\partial_\pi)$ to govern the arithmetic properties of the fibration.  We henceforth focus our  attention on standard fibrations $\pi:X\to \PP^1$ defined over $\QQ$, with the aim of interpreting the growth of the counting function
$
N_{\text{loc}}(\pi,B)$ that was defined in \eqref{eq:count}.
Occasionally we will write $N^\circ_{\text{loc}}(\pi,B)$ for the same counting function, but excluding the finitely many points in the orbifold divisors $\partial_\pi$.

Let us begin by discussing the conjectured power of  $B$ in Conjecture \ref{con2}, which is equal to
\begin{equation}\label{eq:**}
2-\deg \partial_\pi=-\deg(K_{\PP^1,\partial_\pi}),
\end{equation}
where
$K_{\PP^1,\partial_\pi}=K_{\PP^1}+\partial_\pi$. 
The following result relates the geometry of $\pi$ to the geometry of a normalisation of the fibre product of  $\pi$ with a finite cover.

\begin{prop}\label{prop:thin}
Let  $\pi:X \to \mathbb P^1$ be  a standard fibration and let $$\theta \colon \PP^1 \to \PP^1$$ 
be a (possibly ramified)  finite cover of
degree $d$. We define $\pi_\theta:X_\theta\to \PP^1$ to be the normalisation of the fibre product of $\theta$ and $\pi$.
Then we have the following properties.
\begin{enumerate}
\item[(a)] $\pi_\theta \colon X_\theta \to \mathbb P^1$ is a standard fibration.
\item[(b)] 
The orbifold multiplicities $m_{P'}$ for $\pi_{\theta}$ satisfy
\[
m_{P'} \geq \frac{m_P}{e(P'/P)}
\]
for any 
prime divisor $P'$ of $\PP^1$, where  $P = \theta(P')$.
We have equality precisely when
condition (iii) in Definition  \ref{def:orb_mor} is satisfied at $P'$.
\item[(c)] We have
\[
\deg (K_{\mathbb P^1,\partial_{\pi_\theta}} )\geq d \deg (K_{\mathbb P^1,\partial_\pi}),
\]
with  equality precisely when $\theta$ is a finite \'etale orbifold morphism.
\end{enumerate}
\end{prop}

\begin{proof}
\begin{enumerate}
\item[(a)] This is clear from the definition.
\item[(b)]
Consider a component $Z'$ of the fibre of $\pi_\theta$ over a prime divisor $P'$ of $\PP^1$. Suppose that $Z'$ lies over $Z \subseteq X$ and $P'$ lies over $P$. Let $m_{P'}(Z')$ and $m_P(Z)$ denote the multiplicities of these components in their resepective fibres. 
We wish to apply Proposition~\ref{prop:multiplicities under normalisation} 
with $C_1\to C$ being the morphism $\theta:\PP^1\to \PP^1$ and $C_2\to C$ being the morphism 
$\pi:X\to \PP^1$. Then $V\to C_1$ is the morphism $\pi_\theta: X_\theta\to \PP^1$. 
It follows that 
$$
e(Z'/P') = \frac{e(Z/P)}{\gcd\left(e(Z/P),e(P'/P)\right)}.
$$
Hence, since the ramification indices over a codimension one point are precisely
the multiplicities of the different components of the fibre, we obtain
\[
m_{P'}(Z') = \frac{m_P(Z)}{\gcd\!\left(m_P(Z),e(P'/P)\right)}.
\]
Since $m_P(Z) \geq m_P$ and $\gcd\!\left(m_P(Z),e(P'/P)\right) \leq e(P'/P)$ we conclude $m_{P'}(Z') \geq \frac{m_P}{e(P'/P)}$ for all components $Z'$ in the fibre over $P'$.

Clearly, if $m_{P'} = \frac{m_P}{e(P'/P)}$ we have $e(P'/P) \mid m_P$. Now suppose that $e(P'/P) \mid m_P$. To prove the statement we must  show that there is a component $Z'$ over $P'$ with $m_P(Z') = \frac{m_P}{e(P'/P)}$. By the definition of $m_P$ there exists a component $Z$ over $P$ with $m_P=m_P(Z)$. Now let $Z'$ be any component over $P'$ which lies over $P$. Then 
\[
m_{P'}(Z') = \frac{m_P(Z)}{\gcd\!\left(m_P(Z),e(P'/P)\right)} = \frac{m_P}{\gcd\!\left(m_P,e(P'/P)\right)} = \frac{m_P}{e(P'/P)}.
\]
This concludes the proof of part (b).
\item[(c)] 
We will prove the result for orbifolds equipped with a degree $d$ morphism $(C',\partial') \to (C,\partial)$, for general
smooth
 curves $C$ and $C'$,  
in order to distinguish between the two copies of $\mathbb P^1$. The statement is invariant under base change, so we can assume we are working over an algebraically closed field $k = \bar{k}$. We begin by noting that 
$$
\deg K_{C,\partial}
=2g(C)-2+\sum_{P \in C^{(1)}}\left(1-\frac1{m_P}\right) 
$$ and 
$$
\deg K_{C',\partial'}=
2g(C')-2+\sum_{P' \in {C'}^{(1)}}\left(1-\frac1{m_{P'}}\right).
$$
The Riemann--Hurwitz formula yields
\[
2g(C')-2 = d(2g(C)-2) + \sum_{P' \in C'^{(1)}}\left( e(P'/P) - 1\right),
\]
where $P = \theta(P')$. Hence
$$
\deg K_{C',\partial'}=
d(2g(C)-2) + \sum_{P' \in C'^{(1)}}\left( e(P'/P) -\frac1{m_{P'}}\right).
$$
It now follows that 
\begin{align*}
 \deg K_{C',\partial'} &- d\deg K_{C,\partial} 
 \\
 & =\sum_{P' \in {C'}^{(1)}} 
 \left( e(P'/P) -\frac1{m_{P'}}\right)-
 d\sum_{P \in C^{(1)}}\left(1-\frac1{m_P}\right) \\
 & =\sum_{P \in C^{(1)}} 
 \left[\left(\sum_{P'\mid P} e(P'/P)-d\right)
  + \left(\frac d{m_P} - \sum_{P'\mid P} \frac1{m_{P'}} \right)\right].
\end{align*}
Using $\sum_{P' \mid P} e(P'/P) = d$ we see that the first terms all vanish and so 
\[
\deg K_{C',\partial'} - d\deg K_{C,\partial} = \sum_{P \in C^{(1)}} \sum_{P'\mid P} \left(\frac{e(P'/P)}{m_P} -  \frac1{m_{P'}}\right).
\]
This is clearly non-negative by (b), and we have equality if and only if condition (iii) of 
Definition~\ref{def:orb_mor} is satisfied at all $P'$. \qedhere
\end{enumerate}
\end{proof}

In the setting of this result, it follows that the points in $N_{\text{loc}}(\pi,B)$ that are counted by
$N_{\text{loc}}(\pi_\theta, H_\theta,B)$
are expected to contribute at most to the same order of $B$, where $H_\theta$ is the pullback height along $\theta$.
Indeed, in Conjecture \ref{con1} we have
\[
N_{\text{loc}}(\pi_\theta, H_\theta,B) = O_{\ve}\left((B^{\frac{1}{d}})^{\deg(-K_{\mathbb P^1, \partial_{\pi_\theta}})+\ve }\right),
\]
for any $\ve>0$, 
where we use $B^{1/d}$ since $H_\theta$ is an $\mathcal O(d)$-height on $\mathbb P^1$.	
Hence,  in the light of Proposition~\ref{prop:thin}(c),
we should expect no higher order contribution from $N_{\text{loc}}(\pi_\theta, H_\theta,B)$ to $N_{\text{loc}}(\pi,B)$.
Moreover, we should obtain the same exponent of $B$ when $\theta$ is a finite \'etale orbifold morphism.

\medskip

We are now ready to address the possible power of $\log B$.
Let $\pi \colon X \to \mathbb P^1$ be a standard fibration and suppose that 
 $\theta \colon \mathbb P^1 \to \PP^1$
 is a
$G$-torsor of orbifolds under a finite \'etale group scheme $G$ of degree $d$, as presented in Definition~\ref{def:G-torsor}. We
  write $\theta_{v} \colon \mathbb P^1 \to \mathbb P^1$ for the twists of $\theta$ by $v \in 
  \textup{H}^1(\Gal(\bar\QQ/\QQ),G)$. Finally, we shall  write $\pi_{v} \colon X_{v} \to \mathbb P^1$ for the normalisation of the pullback of $\pi$ along $\theta_{v}$.
We are now ready to compare the counting function 
$N_{\text{loc}}^\circ(\pi,B)$ with the counting functions 
$N_{\text{loc}}^\circ(\pi_v,H_v,B)$, for various  $v \in 
  \textup{H}^1(\Gal(\bar\QQ/\QQ),G)$, where 
  $H_v$  
  is the pullback  height along $\theta_v$. (Note that this  is a $\mathcal O(d)$-height on the codomain $\PP^1$ of $\theta_v$.)

\begin{prop}\label{prop:properties}
In the setting above we have the following.
\begin{enumerate}
\item[(a)] A point $x \in \PP^1(\QQ)$ is counted by $N^\circ_{\text{loc}}(\pi,B)$ if and only if there exists $v\in \textup{H}^1(\Gal(\bar\QQ/\QQ),G)$ and $y \in \PP^1(\QQ)$ such that  $\theta_{v}(y)=x$, and such that $y$ is counted by $N^\circ_{\text{loc}}(\pi_{v},H_v,B)$.

\item[(b)] We have
$$
N^\circ_{\text{loc}}(\pi,B) =\frac{1}{\#G(\QQ)} 
\sum_{v\in \textup{H}^1(\Gal(\bar\QQ/\QQ),G)}
N^\circ_{\text{loc}}(\pi_{v},H_v,B).
$$
\item[(c)] Let $\theta_{v}^{-1}(D) = \bigcup_{1\leq i\leq s_D} E^{(i)}_{v}$ be a decomposition into irreducible components, and write $N^{(i)}_{v} = \kappa(E^{(i)}_{v})$ for their function fields. Then 
\[
\Delta(\pi_{v}) = \sum_{D \in (\mathbb P^1)^{(1)}} \sum_{i=1}^{s_D} \left(1-\delta_{D,N_{D,v}^{(i)}}(\pi)\right),
\]
where 
$\delta_{D,N_{D,v}^{(i)}}$ is given by \eqref{eq:dD}.
\item[(d)] The expression $\Delta(\pi_{v})$ only assumes finitely many values.
\end{enumerate}
\end{prop}

\begin{proof}
\begin{enumerate}
\item[(a)] Let $U \subseteq \PP^1$ be the image of the \'etale locus of $\theta$. The restrictions $\theta_v \colon U_v \to U$ are $G$-torsors, and so we have a partition
\[
U(\QQ) = \bigsqcup_{v\in \textup{H}^1(\Gal(\bar\QQ/\QQ),G)}\theta_{v}\left(U_v(\QQ)\right).
\]
Furthermore, the fibre of $\pi_v$ over $y \in U_v(\QQ)$ is isomorphic to the fibre of $\pi$ over $x=\theta_v(y)$. Hence one of these fibres is locally soluble 
precisely when the other is. Finally, since $\theta \colon \PP^1 \to \PP^1$ has degree $d$, the pullback of the $\mathcal O(1)$-height pulls back to a $\mathcal O(d)$-height. 
\item[(b)] This follows from the partition in (a), and the fact that each fibre has $\#G(\QQ)$ points. 
\item[(c)] This directly follows from the definition of $\delta_{D,N}$ and $\pi_v$.
\item[(d)] This follows from Theorem \ref{t:remark}. \qedhere
\end{enumerate}
\end{proof}

In the setting of Theorem \ref{t:2} we  consider $\mu_d$-covers parametrised by 
$\QQ^\times/\QQ^{\times,d}$. The following result therefore follows from part (c) of Proposition \ref{prop:properties}.

\begin{cor}\label{c:goat}
We have 
$
\Delta(\pi_v)=\Theta_v(\pi)
$ 
in \eqref{eq:theta-v}.
\end{cor}

In principle there might be infinitely many twists $\pi_{v}$ for which $\Delta(\pi_{v})$ differs from the expected exponent  $\Delta(\pi)$ defined in \eqref{eq:Delta}.
The following example illustrates an instance where the points counted by the covers for which $\Delta(\pi_{v})=\Delta(\pi)$ can form a non-trivial cothin set in $\PP^1(\QQ)$.

\begin{example}
Consider the fibration $\pi \colon X \to \PP^1$ with three double fibres over $0$, $-1$ and $\infty$, together with  precisely one other non-split fibre over $1$, which has multiplicity one and  is split by a quadratic extension $K/\QQ$.
Let $C_{v}$ be the conic 
$$
v_1x_1^2+v_2x_2^2=x_0^2
$$ 
in $\mathbb P^2$ defined by $v=(v_1,v_2) \in \QQ^\times/\QQ^{\times,2}\times \QQ^\times/\QQ^{\times,2}$. We apply the partition in part (b) of Proposition \ref{prop:properties}
 with the full family of twists
\[
\theta_v: C_{v} \to \PP^1, \quad [x_0\colon x_1 \colon x_2] \mapsto [v_1x_0^2\colon v_2x_1^2].
\]
This is the finest partition in the sense of 
Remark \ref{rem:cap}, since we have $\pi_1^{\text{orb}}(\PP^1,\partial_\pi) = \ZZ/2\ZZ \times \ZZ/2\ZZ$, and any  $\theta_{v}$ is geometrically a universal orbifold cover.
Consider the fibres $\theta_{v}^{-1}(1)$ as $v$ varies, which on algebras are biquadratic \'etale $\QQ$-algebras $\prod_i N^{(i)}_{1,v}$. Infinitely many of these contain the splitting field $K$ of the fibre and for such $v$ we have
\[
1-\delta_{1,\mathbb Q} < 1  = 1-\delta_{1,N_{1,v}^{(\alpha)}}(\pi) < \sum_i \left(1-\delta_{1,N_{1,v}^{(i)}}(\pi)\right),
\]
where $\alpha$ is such that $K \subseteq N^{(\alpha)}_{1,v}$.
However, each of these infinitely many $(\ZZ/2\ZZ \times \ZZ/2\ZZ)$-covers factors through only two $\ZZ/2\ZZ$-covers. Hence, the set of points counted through the $v$ for which
\[
1-\delta_{1,\mathbb Q} \ne \sum_i \left(1-\delta_{1,N_{1,v}^{(i)}}(\pi)\right),
\]
is a thin set.
In the case of a non-trivial Galois action on the components of the multiple fibres, we will need to deal with them in a similar manner to conclude that the points counted in the covers $\theta_v$ for $\Delta(\pi_v) \ne \Delta(\pi)$ form a thin set.
\end{example}


\section{A sparsity criterion}\label{s:sparsity}

Let $k$ be a number field. Let $X$ and $Y$ be smooth, proper varieties over $k$, and let $D$ and $E$ be strict normal crossings divisors on $X$ and $Y$ respectively, where $f^{-1}(E) \subseteq D$. Assume that the induced morphism $f: (X, D) \to (Y,E)$ is a \emph{toroidal} morphism; i.e., a toroidal morphism between toroidal embeddings, or equivalently, a log smooth morphism of (Zariski) log regular schemes. Fix $Q \in Y(k)$. We want to understand when $f^{-1}(Q)$ is everywhere locally soluble.

Let $S$ be a finite set of places including all {\em places of bad reduction for $f$}. This means that we have a good model $\overline{f}: (\mathcal{X}, \mathcal{D}) \to (\mathcal{Y},\mathcal{E})$ for $f$ over $\mathcal{O}_{k,S}$ with the property that $\overline{f}^{-1}(\mathcal{E}) \subseteq \mathcal{D}$, such that $(\mathcal{X}, \mathcal{D})$ and $(\mathcal{Y},\mathcal{E})$ are still log regular, and such that $\overline{f}$ is still log smooth with respect to the divisorial log structures induced by $\mathcal{D}$ and $\mathcal{E}$.

Let $v \not\in S$ be a finite place of $k$. Let $\mathcal{Q}_v \in \mathcal{Y}(\mathcal{O}_v)$ be the unique lift of $Q \in Y(k)$ to an $\mathcal{O}_v$-point. We will give necessary and sufficient conditions for the existence of an $\mathcal{O}_v$-point $\mathcal{P}_v$ on $\mathcal{X}$ such that $f(\mathcal{P})_v = \mathcal{Q}_v$, for $v \gg 0$.

If $\mathcal Q_v \not \subseteq \mathcal E$, then the $\mathcal{O}_v$-point $\mathcal{Q}_v$ can be seen as a morphism $$\mathcal{Q}_v \colon (\mathrm{Spec}\,\mathcal{O}_v)^\dagger \to (\mathcal{Y},\mathcal{E})$$ of log schemes, where $(\mathrm{Spec}\,\mathcal{O}_v)^\dagger$ is the scheme $\mathrm{Spec}\,\mathcal{O}_v$ equipped with the divisorial log structure induced by the closed point. This morphism induces a morphism of associated \emph{Kato fans} $$F(\mathcal{Q}_v) \colon \mathrm{Spec}\,\mathbb{N} \cong F((\mathrm{Spec}\,\mathcal{O}_v)^\dagger) \to F(\mathcal{Y},\mathcal{E}).$$ 
In other words, we get an $\mathbb{N}$-valued point $F(\mathcal{Q}_v) \in F(\mathcal{Y},\mathcal{E})(\mathbb{N})$.

 If $\mathcal{Q}_v$ is the image of $\mathcal{P}_v \in \mathcal{X}(\mathcal{O}_v)$, then clearly $F(\mathcal{Q}_v)$ cannot lie anywhere in $F(\mathcal{Y},\mathcal{E})(\mathbb{N})$; it needs to be an element of the potentially smaller set 
 $$
 \mathrm{image}\left(F(\mathcal{X},\mathcal{D})(\mathbb{N}) \to F(\mathcal{Y},\mathcal{E})(\mathbb{N})\right).
 $$ This means that if $F(\mathcal{Q}_v)$ does \emph{not} lie in the image of $F(\mathcal{X},\mathcal{D})(\mathbb{N})$, then surely $\mathcal{Q}_v$ cannot lift to a $\mathcal{O}_v$-point on $\mathcal{X}$. This is a {\em sparsity criterion} in the sense of \cite[\S 2]{LS}, but still a rather na\"ive one, since it does not take important arithmetic information into account.
 
\begin{defin}
Let $\bar{\mathcal P}_v$ be an $\mathbb F_p$-point on $\mathcal X$. With the notation above, we define $F(\mathcal{X},\mathcal{D})(\mathbb{N})_{\bar{\mathcal P}_v}$ as the subset of $F(\mathcal{X},\mathcal{D})(\mathbb{N})$ with the property that $\bar{\mathcal{P}}_v$ lies is the logarithmic stratum associated to the image of the closed point $\mathbb{N}_{>0}$ of $\mathrm{Spec}\,\mathbb{N}$.
\end{defin} 

\begin{prop}
With notation as above, let $\bar{\mathcal P}_v$ be an $\mathbb F_p$-point on $\mathcal X_{\mathcal Q_v}$ and assume that $F(\mathcal{Q}_v)$ does \emph{not} lie in 
$$
\mathrm{image}\left(
F(\mathcal{X},\mathcal{D})(\mathbb{N})_{\bar{\mathcal P}_v} \to F(\mathcal{Y},\mathcal{E})(\mathbb{N})\right).
$$ 
Then $\bar{\mathcal P}_v \in \mathcal X_{\mathcal Q_v}(\mathbb F_p)$ does not lift to $\mathcal{P}_v \in \mathcal{X}_{\mathcal Q_v}(\mathcal{O}_v)$.
\end{prop}

\begin{proof}
Assume that $\bar{\mathcal P}_v$ lifts, i.e., $\mathcal{Q}_v = \overline{f}(\mathcal{P}_v)$ for some $\mathcal{P}_v \in \mathcal{X}(\mathcal{O}_v)$ with $\bar{\mathcal P}_v = \mathcal{P}_v \mod{v}$ (which is the image of $\mathrm{Spec}\,\mathbb{F}_v$ under $\mathcal{P}_v$).  Therefore the image of $F(\mathcal{P}_v) \in F(\mathcal{X},\mathcal{D})(\mathbb{N})$ under the map $F(\mathcal{X},\mathcal{D})(\mathbb{N}) \to F(\mathcal{Y},\mathcal{E})(\mathbb{N})$ comes from $F(\mathcal{X},\mathcal{D})(\mathbb{N})_{\bar{\mathcal P}_v}$, as desired. \end{proof}
 
In fact, the logarithmic Hensel lemma \cite[Proposition 5.13]{LSS} yields more, as in the following result.

\begin{prop}\label{prop:loghensel}
If $\bar{\mathcal P}_v$ is an $\mathbb F_v$-point on $\mathcal X_{\mathcal Q_v}$, the following are equivalent:
\begin{itemize}
\item[(a)] $\bar{\mathcal P}_v$ lifts to an $\mathcal{O}_v$-point on $\mathcal{X}_{\mathcal Q_v}$;
\item[(b)] $F(\mathcal{Q}_v) \in \mathrm{image}\left(F(\mathcal{X},\mathcal{D})(\mathbb{N})_{\bar{\mathcal P}_v} \to F(\mathcal{Y},\mathcal{E})(\mathbb{N})\right)$.
\end{itemize}
\end{prop}
 
\begin{proof}
Since we have already shown that (a) implies (b), it remains to prove the reverse implication. This is an application of \cite[Proposition 5.13]{LSS}.  Indeed, let $s^\dagger = \mathrm{Spec}\,\mathbb{F}_v$ with the standard log structure of rank $1$, and $S^\dagger = \mathrm{Spec}\,\mathcal{O}_v$. Let $j: s^\dagger \to S^\dagger$ be the canonical closed immersion.

By assumption there is an element $p_v \in F(\mathcal{X},\mathcal{D})(\mathbb{N})_{\bar{\mathcal P}_v}$ which maps to $F(\mathcal{Q}_v) \in F(\mathcal{Y},\mathcal{E})(\mathbb{N})$, and there is an $\mathbb F_v$-point $u \colon \mathrm{Spec}\ \mathbb F_v \to X$ on the associated stratum of $(\mathcal X, \mathcal D)$. We can uniquely make $u$ into a morphism of log schemes $s^\dagger \to (\mathcal X, \mathcal D)$ such that $F(u) = p_v$ under the identification $F(\mathbb N) = F(s^\dagger)$, similar to the proof of Proposition~6.1 in \cite{LSS}.

Since $F(f)$ maps $F(u)$ to $F(\mathcal Q_v)$ we have a commutative diagram:

\begin{center}
\begin{tikzpicture}[auto]
\node (L1) {$s^\dagger$};
\node (L2) [below= 0.9cm of L1] {$S^\dagger$};
\node (M1) [right= 0.9cm  of L1] {$(\mathcal{X},\mathcal{D})$};
\node (M2) at (M1 |- L2) {$(\mathcal{Y},\mathcal{E})$};
\draw[->] (L1) to node {\footnotesize $u$} (M1);
\draw[->] (L2) to node [below] {\footnotesize $\mathcal{Q}_v$} (M2);
\draw[->] (L1) to node [left] {\footnotesize $j$}  (L2);
\draw[->] (M1) to node {\footnotesize $\overline{f}$} (M2);
\end{tikzpicture}
 \end{center}
Now \cite[Proposition 5.13]{LSS} provides a lift $S^\dagger \to (\mathcal{X},\mathcal{D})$ of $\mathcal{Q}_v$. The morphism of schemes which underlies this lift is the $\mathcal{O}_v$-point $\mathcal{P}_v$ we are looking for.
\end{proof} 

Using this statement we can give precise conditions for locally solubility. We allow ourselves to work over a general number field $k/\QQ$ and so define a {\em standard fibration} to be a 
 a dominant morphism $\pi:X\to \PP^1$  with geometrically integral generic fibre, such that $X$ is  
 a smooth, proper, geometrically irreducible $k$-variety.

Let   $E$ be the reduced divisor of $\mathbb P^1$ of the non-split fibres of $\pi$. Let $D$ be the reduced  divisor underlying $\pi^{-1} (E)$. By embedded resolutions of singularities, there exists a birational morphism $X' \to X$ such that the pullback $D'$ of $D$ has strict normal crossings. Since $X\setminus D \cong X' \setminus D'$ over $\mathbb P^1$ we see that $N_{\text{loc}}(\pi',B)$ differs by a constant from $N_{\text{loc}}(\pi,B)$, where $\pi' \colon X' \to X \to Y$ is the composition. 
 Thus, for the purposes of upper and lower bounds, we  can assume without loss of generality that the reduced subschemes of the non-split fibres of $\pi$ have strict normal crossings.

\begin{thm}\label{thm:lifting}
Let $X \to \mathbb P^1$ be a standard fibration whose non-split fibres in their reduced subscheme structure are sncd. There exists a finite set of primes $S$ and  a model $\mathcal X \to \mathbb P^1_{\mathcal O_S}$ such that the following holds for $v \not\in S$.  Fix a point $\mathcal Q \in \mathbb P^1(\mathcal O_S)$ for which the fibre $X_Q$ is split. Then any $\mathbb F_v$-point $\bar{\mathcal P}_v \in \mathcal X_{\mathcal Q}(\mathbb F_v)$  lifts to a point $\mathcal P_v \in \mathcal X_{\mathcal Q}(\mathcal O_v)$ precisely if
for every closed point $V(h) \in (\mathbb P^1)^{(1)}$ we have that $v(h(\mathcal Q))$ lies in the positive linear span of the multiplicities $m_i$ of the components of $\mathcal X_{V(h)}$ that contain $\bar{\mathcal P}_v$.
\end{thm}

Note that the last condition is trivially satisfied for all closed points $V(h)$ for which $v(h(\mathcal Q))=0$, and also for those for which if $X_{V(h)}$ is split. By restricting $S$ further, we can assume that there is at most one non-split fibre $X_{V(h)}$ for which we have to check this condition.

\begin{proof}[Proof of Theorem \ref{thm:lifting}]
By the definition of $D$ on $X$ and  $E$ on  $\mathbb P^1$, 
we see that $(X,D) \to (\mathbb P^1,E)$ is log smooth. For a suitable finite set of primes $S$ this extends to $\mathcal O_S$-schemes and divisors, such that $\mathcal D \subseteq \mathcal X$ and $\mathcal E \subseteq \mathbb P^1_{\mathcal O_S}$ still have strict normal crossings and $(\mathcal X,\mathcal D) \to (\mathbb P^1_{\mathcal O_S},\mathcal E)$ is also log smooth. We will check that this model satisfies the condition.

Consider $\bar{\mathcal P}_v \in \mathcal X_{\mathcal Q}(\mathbb F_v)$ and let $V(h) \subset \mathbb P^1_{\mathcal O_S}$ be the unique non-split fibre containing $\bar{\mathcal Q}_v = \pi(\bar{\mathcal P}_v)$. Suppose that we can write $v(h(\mathcal Q)) = \sum_i a_i m_i$ with $a_i > 0$ integers, and $m_i$ the multiplicities of the $r$ components of $X_{v(h)}$ which contain $\bar{\mathcal P}_v$. Around $\bar{\mathcal P}_v$ and $\bar{\mathcal Q}_v$ the Kato fans have affine charts $\mathbb N^r$ and $\mathbb N$. Under this identification we have $F(\mathcal Q_v) = v(h(\mathcal Q)) \in \mathbb N$ and $F(\mathcal{X},\mathcal{D})(\mathbb{N}) \to F(\mathbb P^1_{\mathcal O_S},\mathcal{E})(\mathbb{N})$ is given by $(u_i) \mapsto \sum m_iu_i$. Hence the result follows from Proposition~\ref{prop:loghensel}.
\end{proof}

\begin{rem}
In \cite[\S~2]{LS} the following was proven: if $v(h(\mathcal Q))=1$ then $\mathcal X_{\mathcal Q}$ is a regular scheme. This implies that any $\mathbb F_v$-point on $\mathcal X_{\mathcal Q}$ which lies on the intersection of at least two components of the reduction $\mathcal X_{\mathcal Q,v}$ does not lift to a $\mathbb Q_v$-point on $\mathcal X_{\mathcal Q}$. This last statement directly follows from our criterion above, since then the valuation $v(h(\mathcal Q))=1$ cannot possibly lie in the positive linear span of two positive integers.
\end{rem}

The above conditions make it easy to check if an $\mathbb F_v$-point lifts. However, one cannot deduce the existence of $\mathbb F_v$-points purely from valuations and multiplicities, as explained by Loughran and Matthiesen \cite[Lemma 6.2]{LM}. In general, this only allows us to give necessary conditions for local solubility.

\begin{cor}\label{cor:sparsity}
Let $X \to \mathbb P^1$ be a standard fibration and let $Q\in \PP^1(k)$. 
Suppose that  $X_Q(k_v) \ne \emptyset$
for $v \not \in S$.
Then for every closed point $D=V(h) \in (\mathbb P^1)^{(1)}$, we have either $v(h(Q)) > m_D$,  or else $v(h(Q)) = m_D$ and $v$ belongs to
\[
T_D =\left\{v\not\in S: \text{$\frob_v$ fixes an element of $S_{D}$} \right\}.
\]
(Recall that  $S_D$ is the set of geometric components of $X_D$ of minimum multiplicity $m_D$.)
\end{cor}

In the special case that the non-split fibres all lie about $k$-rational points in $\PP^1$, we can (after possibly extending the set $S$ again) make this even more  precise, as follows. 

\begin{cor}\label{cor:onlysplitoverdegree1}
Let $X \to \mathbb P^1$ be a standard fibration and let $Q\in \PP^1(k)$. 
Assume that the non-split fibres of $X \to \mathbb P^1$ all lie above $k$-rational points. Then $X_Q(k_v) \ne \emptyset$ precisely if for every $V(h) \in (\mathbb P^1)^{(1)}$ the fibre $X_{V(h)}$ has intersecting geometric components of multiplicity $m_i$ which are fixed by $\frob_v$, such that $v(h(Q))$ lies in the positive linear span of the $m_i$.
\end{cor}

\begin{proof}
We will start with $S$ and $\mathcal X \to \mathbb P^1_{\mathcal O_S}$ as above. By the results above we have that $P_v \in X_Q(k_v)$ reduces to an $\mathbb F_v$-point on $\mathcal X$. Since this $\mathbb F_v$-points lifts we get the result.

For the inverse implication we will need to enlarge $S$, as follows. Firstly we do so to assume that all fibres of $\mathcal X \setminus \mathcal D \to \mathbb P^1_{\mathcal O_S} \setminus \mathcal E$ are geometrically integral. Hence by Lang--Weil we find a smooth $\mathbb F_v$-point all those fibres except for finitely many $v$.
Now let $W$ be a geometric component of a non-open stratum of $(X,D)$, which is defined over $k'/k$. The closure $\mathcal W$ of $W$ will have geometrically irreducible fibres over all but finitely many places of $k'$. Hence after enlarging $S$ we see that $\mathcal W$ has an $\mathbb F_{v'}$-point for all $v' \mid v$, for $v \not \in S$. Since there are only finitely many strata and each has again finitely many components we can enlarge $S$ to make this true for all possible $W$.

Suppose now that $\frob_v$ fixes the components of $D$ which define the stratum containing $W$. Then for any $v' \mid v$ we see that $W$ contains an $\mathbb F_{v'}=\mathbb F_v$-point. We can lift this point under the conditions in Theorem~\ref{thm:lifting}.
\end{proof}

\section{Multiple fibres via the large sieve}\label{s:large}

We place ourselves in the setting of   Theorems \ref{t:1} and \ref{t:2}. 
Let $\pi \colon X \to \PP^1$
be a standard fibration with orbifold  divisor 
$$
\partial_\pi = \left(1-\frac1{m_0}\right)[0] + \left(1-\frac1{m_\infty}\right)[\infty],
$$ 
in the notation of \eqref{eq:def_partial},
for $m_0,m_\infty\in \NN$.  Note that $2-\deg\partial_\pi=\frac{1}{m_0}+\frac{1}{m_\infty}$.
We define  $d=\gcd(m_0,m_\infty)$.
We shall apply the theory from 
Section \ref{sec:orb} to the family of $\mu_d$-torsors 
$$
\theta_{v} \colon \PP^1 \to \PP^1, \quad [x_0\colon x_1] \to [v_0x_0^d \colon v_1 x_1^d],
$$
which are parametrised by $v = v_1/v_0 \in \mathbb Q^\times/\QQ^{\times,d}=\text{H}^1(\Gal(\bar\QQ/\QQ),\mu_d)$. 
Let 
$\pi_{v} \colon X_{v} \to \PP^1$ 
be the  normalisation of the pullback of  $\pi$ along $\theta_{v}$.

The main result of this section is the following, which pertains to the density of locally soluble fibres on the standard fibration $\pi_v:X_v\to \PP^1$,
relative to the pullback height  $H_v$ along $\theta_v$.
 We denote by $\rad(n)=\prod_{p\mid n}p$, the square-free radical of any  $n\in \NN$.

\begin{prop}\label{prop:2fibres}
Let  $\ve >0$ and let
$v = v_1/v_0 \in \mathbb Q^\times/\QQ^{\times,d}$.
Then 
$$
 N_{\text{loc}}(\pi_{v},H_v,B) \ll_{\ve}
 c_{v,\ve} B^{\frac{1}{m_0}+\frac{1}{m_\infty}},
 $$
 where
\begin{equation}\label{eq:cv}
c_{v,\ve} =\frac{|v_0v_1|^\ve}{\rad(v_0)|v_0|^{1/m_0}\rad(v_1)|v_1|^{1/m_\infty}}.
\end{equation}
Furthermore, if  $|v_0v_1|\leq B^\ve$, then 
\begin{align*}
 N_{\text{loc}}(\pi_{v},H_v, B) 
 \ll_{\ve} 
 c_{v,\ve} \frac{B^{\frac{1}{m_0}+\frac{1}{m_\infty}}}{(\log B)^{\Theta_v(\pi)}},
\end{align*}
where $\Theta_v(\pi)$ is given by \eqref{eq:theta-v}.
\end{prop}

We shall begin the proof of this result in Section \ref{s:prelim-large}. Our argument is based on the large sieve, which is  recalled in Section
 \ref{s:large'}.  Taking the result on faith for the moment, we proceed to show how it can be used to establish 
Theorems~\ref{t:1} and \ref{t:2}.

\begin{rem}\label{rem:zwiebel}
Proposition 
\ref{prop:2fibres} is consistent with Conjecture \ref{con2} for a fixed choice of
$v \in \mathbb Q^\times/\QQ^{\times,d}$.
Indeed, we have $H_v(x)=H(x)^d$ where $H(x)$ is an $\mathcal O(1)$-height on $\mathbb P^1$.
It follows that
$$
 N_{\text{loc}}(\pi_{v},B) =
 N_{\text{loc}}(\pi_{v},H,B) =
  N_{\text{loc}}(\pi_{v},H_v, B^d) 
 \ll_{v}  \frac{B^{\frac{d}{m_0}+\frac{d}{m_\infty}}}{(\log B)^{\Theta_v(\pi)}}.
$$
The   orbifold base of $\pi_v$ is $(X_v,\partial_{\pi_{v}})$,
with 
$$
\partial_{\pi_{v}} = \left(1-\frac d{m_0}\right)[0] + \left(1-\frac d{m_\infty}\right)[\infty],
$$
by part (b) of 
Proposition \ref{prop:thin}. It  follows from \eqref{eq:**} 
and part (c) of Proposition~\ref{prop:thin} that 
that $\frac{d}{m_0}+\frac{d}{m_\infty}=-d \deg K_{\pi,\partial_\pi}=2-\deg \partial_{\pi_v}$.
Moreover, 
$\Theta_v(\pi)=\Delta(\pi_{v})$, by part (c) of Proposition \ref{prop:properties}.
\end{rem}

\begin{proof}[Proof of Theorem \ref{t:1}]
In this case there is only one multiple fibre above $0$ and so $m_\infty=1$ and  $d=1$.
Thus $\text{H}^1(\Gal(\bar\QQ/\QQ),\mu_d)$ is the trivial group
and it follows directly from 
Proposition 
\ref{prop:2fibres} that
$$
 N_{\text{loc}}(\pi,B) 
 \ll \frac{B^{\frac{1}{m_0}+1}}{(\log B)^{\Theta_1(\pi)}}.
$$
We have already seen that $\frac{1}{m_0}+1=2-\deg \partial_\pi$. Moreover, we saw that $\Theta_1(\pi)=\Delta(\pi)$ in \eqref{eq:check}. 
\end{proof}

\begin{proof}[Proof of Theorem \ref{t:2}]
We appeal to the  decomposition in part (b) of Proposition~\ref{prop:properties}. 
This gives 
$$
 N_{\text{loc}}(\pi,B) \ll \sum_{
 v = v_1/v_0 \in \mathbb Q^\times/\QQ^{\times,d}} N_{\text{loc}}(\pi_v,H_v,B).
$$
For any $\delta>0$ we clearly have 
\begin{align*}
\sum_{n>x} \frac{1}{\rad(n)n^{\delta} }<
\sum_{n=1}^\infty \frac{(n/x)^{\delta/2}}{\rad(n)n^{\delta} }
&=
x^{-\delta/2}\prod_{p} \left(1+\sum_{k=1}^\infty \frac{1}{p^{1+k\delta/2}}\right)\\
&\ll_\delta x^{-\delta/2}.
\end{align*}
Let $\ve>0$. 
In the light of the latter bound, it follows from the first part of Proposition \ref{prop:2fibres}
that there exists $\delta(\ve)>0$, such that the terms with $|v_0v_1|>B^\ve$ make an overall contribution $O_{\ve}(B^{1/m_0+1/m_{\infty}-\delta(\ve)})$ to 
$ N_{\text{loc}}(\pi,B)$. For the terms with 
$|v_0v_1|\leq B^\ve$, we apply the second part of 
Proposition \ref{prop:2fibres}. 

This easily leads to the conclusion that 
$$
 N_{\text{loc}}(\pi,B) \ll_{\ve} 
 B^{\frac{1}{m_0}+\frac{1}{m_\infty}-\delta(\ve)}+
 \sum_{|v_0v_1|\leq B^\ve}
 c_{v,\ve} \frac{B^{\frac{1}{m_0}+\frac{1}{m_\infty}}}{(\log B)^{\Theta_v(\pi)}}
 \ll_{\ve}
  \frac{B^{\frac{1}{m_0}+\frac{1}{m_\infty}}}{(\log B)^{\Theta(\pi)}},
 $$
where $\Theta(\pi)$ is given by \eqref{eq:put}.
The  statement of the theorem  follows, since we have already remarked that 
 $\frac{1}{m_0}+\frac{1}{m_\infty}=2-\deg\partial_\pi$.
\end{proof}

\subsection{The large sieve}\label{s:large'}

We begin by stating the version of the large sieve that we shall use in this paper. 

\begin{lemma}\label{large}
Let $m\in \NN$, let $B_0,B_1\geq 1$ and let $\Omega\subset \ZZ^{2}$. For each prime $p$ assume that there exists $\overline{\omega}(p)\in [0,1)$ such that the reduction modulo $p^m$ of $\Omega$ has cardinality at most
$(1-\overline{\omega}(p))p^{2m}$. Then 
$$
\#\left\{\x\in \Omega: |x_i|\leq B_i, \text{ for $i=0,1$}\right\}\ll
\frac{ (B_0 +Q^{2m})(B_1 +Q^{2m})}{L(Q)},
$$
for any $Q\geq 1$,
where
$$
L(Q)=\sum_{q\leq Q} \mu^2(q) \prod_{p\mid q} \frac{\overline{\omega}(p)}{1-\overline{\omega}(p)}.
$$
\end{lemma}

\begin{proof}
When $m=1$ this is a straightforward rephrasing of the multidimensional large sieve worked out by Kowalski \cite[Thm.~4.1]{large}. The extension to $m>1$ is routine and will not be explained  here. 
\end{proof}

\subsection{Preliminary steps}\label{s:prelim-large}

Recall that $d=\gcd(m_0,m_\infty)$.
Henceforth, we usually  write 
$\v=(v_0,v_1)\in \ZZp^2$ for the point 
$v = v_1/v_0 \in \mathbb Q^\times/\QQ^{\times,d}$. 
We may clearly  proceed under the assumption that 
$v_0,v_1$ are both free of $d$th powers.

Let $S$ be a large enough finite set of primes, 
as required for the arguments in Section \ref{s:sparsity} to go through. 
Suppose that $E_1,\dots,E_r\in (\PP^1)^{(1)}$ are the closed points distinct from $0$ and $\infty$, 
where $\pi_v$ is not smooth. For each $1\leq j\leq r$, 
assume that $E_j=V(h_j)$ for a square-free binary form $h_j\in \ZZ_S[x_0,x_1]$.
We may  further assume that 
 $h_j$ is irreducible over $\QQ$ and 
 coprime to the monomial $x_0x_1$, and that 
  the coefficients of $h_j$ are relatively coprime.

We proceed by  defining  the sets
\begin{align*}
T_0&=\left\{p\not\in S: \text{$\frob_p$ fixes an element of $S_{0}$} \right\},\\
T_\infty&=\left\{p\not\in S: \text{$\frob_p$ fixes an element of $S_{\infty}$} \right\}, \\
U_j&=\left\{p\not\in S: \text{$\frob_p$ fixes an element of $S_{E_j}$} \right\},
\end{align*}
for $1\leq j\leq r$.
The fibre $X_{v,y}$ of the fibration $\pi_{v} \colon X_{v} \to \mathbb P^1$ has a $\mathbb Q_p$-point precisely if $X_{\pi_{v}(y)}$ does and thus  we can apply the sparsity conditions Corollary~\ref{cor:sparsity}. This yields the upper bound
$ N_{\text{loc}}(\pi_{v},B) \leq
 M_{\mathbf{v}}(B)$,
where   $M_{\mathbf{v}}(B)$ is defined to be the number of 
$\mathbf y =(y_0,y_1)\in \ZZ^2$ such that 
$\gcd(v_0y_0,v_1y_1)=1$ and 
$\max\{|v_0y_0^d|,|v_1y_1^d|\}\leq B$,
with 
$$ 
p\not \in S \Rightarrow 
\begin{cases}
\left(v_p(x_0) = m_0 \text{ and } p \in T_0\right), \text{ or } v_p(x_0) > m_0, \\ \left(v_p(x_1) = m_\infty \text{ and } p \in T_\infty\right), \text{ or } v_p(x_1) > m_\infty, \\ p \| h_j(\mathbf x) \Rightarrow p \in U_j,
\end{cases}
$$
where $(x_0,x_1)=(v_0y_0^d,v_1y_1^d)$. 
 Write 
$v_0=a_{0}w_0'$ 
and
 $y_0=b_0z_0$, where 
 $w_0'z_0$ is coprime to all the primes in $S$ and 
 $p\mid a_0b_0 \Rightarrow p\in S$.
 Let $p\not \in S$. 
 Then $v_p(w_0'z_0^d)=m_0$ if and only if 
 $v_p(w_0')=0$ and 
 $v_p(z_0)=m_0/d$, since 
$w_0'$ is  free of $d$th powers. Similarly, if $v_p(w_0'z_0^d)>m_0$ then 
either $v_p(z_0)>m_0/d$, or else  $v_p(z_0)=m_0/d$ and $p\mid w_0'$.
 This suggests that we may write
 $$
 v_0=a_0w_0, \quad y_0=b_0 s_0^{m_0/d} t_0^{m_0/d}u_0,
 $$
 where 
 \begin{itemize}
 \item $p\mid a_0b_0 \Rightarrow p\in S$;
 \item
$p\mid s_0w_0u_0 \Rightarrow p\not \in S$;
\item $s_0,t_0$ square-free;
\item $p\mid w_0 \Rightarrow p\mid s_0$;
\item  
$p\mid t_0 \Rightarrow p\in T_0$; and
\item  $u_0$ is $(m_0/d+1)$-full. 
 \end{itemize}
Similarly, we have a factorisation
$$
 v_1=a_1w_1, \quad y_1=b_1 s_1^{m_\infty/d} t_1^{m_\infty/d}u_1,
 $$
 where 
 \begin{itemize}
 \item $p\mid a_1b_1 \Rightarrow p\in S$;
 \item
$p\mid s_1w_1u_1 \Rightarrow p\not \in S$;
\item $s_1,t_1$ square-free;
\item $p\mid w_1 \Rightarrow p\mid s_1$;
\item  
$p\mid t_1 \Rightarrow p\in T_\infty$; and
\item  $u_1$ is $(m_\infty/d+1)$-full. 
 \end{itemize}

There are $O_\ve(|v_0v_1|^\ve) $ choices for $a_i,s_i,w_i\in \ZZ$ for $i=0,1$, by the standard estimate for the divisor function.  We fix a choice of $b_0,b_1, u_0,u_1$ and 
write 
\begin{equation}\label{eq:factor}
A_0=a_0b_0^d s_0^{m_0} u_0^d 
w_0  \quad \text{ and } \quad A_1=a_1b_1^d s_1^{m_\infty} u_1^d 
w_1.
\end{equation}
Note that we have $\gcd(A_0,A_1)=1$. 
Moreover, let
$$
R_0=\left(\frac{B}{|A_0|}\right)^{1/m_0}, \qquad 
R_1=\left(\frac{B}{|A_1|}\right)^{1/m_\infty},
$$
and 
\begin{equation}\label{eq:hj-u}
g_j(\t)=h_j(A_0t_0^{m_0},A_1 t_1^{m_\infty}), \text{ for $1\leq j\leq r$}.
\end{equation}
The binary form $g_j(t_0,t_1)$ is square-free and coprime to the monomial $t_0t_1$, since $h_j(t_0,t_1)$ satisfies these properties.  
For a
(possibly infinite) set $T$ of primes, let 
$$
\1_T(n)=\begin{cases}
1 &\text{ if $p\mid n \Rightarrow p\in T$,}\\
0 &\text{ otherwise.}
\end{cases}
$$
In what follows, we will also write   $T^c$ to denote the   complement of $T$ in the full set of primes.

Then, with all this notation in mind, we have 
$$
M_\v(B)\ll
 \sum_{v_0=a_0w_0}\sum_{v_1=a_1w_1}  \sum_{\substack{b_0,b_1\\
p\mid b_0b_1 \Rightarrow p\in S}
} \sum_{u_0,u_1\in \ZZ} L(R_0,R_1),
$$
where
\begin{equation}\label{eq:LRR}
L(R_0,R_1)=
\sum_{\substack{
(t_0,t_1)\in \ZZ^2\\
|t_0|\leq R_0, ~|t_1|\leq R_1}}
\mu^2(t_0t_1)
\1_{T_0}(t_0)
\1_{T_\infty}(t_1)
\prod_{j=1}^r
 \1^\sharp_{U_j}\left(t_0,t_1\right),
\end{equation}
and where
$$
\1_{U_{j}}^\sharp(t_0,t_1)=
\begin{cases}
1 & \text{
 if  $p\| g_j(\t) \Rightarrow 
 p\in U_j$},\\
0 & \text{ otherwise.}
\end{cases}
$$
The trivial bound for $L(R_0,R_1)$ is 
\begin{align*}
L(R_0,R_1)
&\ll \frac{B^{1/m_0+1/m_\infty}}{ |A_0|^{1/m_0} |A_1|^{1/m_\infty}}\\
&\ll \frac{B^{1/m_0+1/m_\infty}}{ |s_0||v_0|^{1/m_0} |s_1||v_1|^{1/m_\infty} |b_0u_0|^{d/m_0}|b_1u_1|^{d/m_\infty}},
\end{align*}
by \eqref{eq:factor}.
Clearly 
\begin{equation}\label{eq:rad}
|s_i|\gg \rad(v_i), \quad \text{for $i=0,1$}, 
\end{equation}
for a suitable implied constant depending only on $S$.
Note that 
$$
\sum_{\substack{|b_0|>J\\ p\mid b_0\Rightarrow p\in S}} |b_0|^{-d/m_0} \ll \frac{1}{J^{d/m_0}},
$$
for any  $J\geq 1$. Similarly, 
$$
\sum_{\substack{|u_0|>J\\ 
\text{$u_0$ is $(m_0/d+1)$-full}}}
 |u_0|^{-d/m_0} \ll \frac{1}{J^{d^2/(m_0(m_0+d))}},
$$

Let $\ve>0$. In what follows it will be convenient to recall the notation \eqref{eq:cv} for $c_{v,\ve}$ in the statement of Proposition \ref{prop:2fibres}.
It now follows that the overall contribution to
$M_\v(B)$ from parameters $b_0,u_0$ in the range  $\min(|b_0|,|u_0|)>B^\ve$, 
or parameters $b_1,u_1$ in the range  $\min(|b_1|,|u_1|)>B^\ve$
is clearly 
$$
\ll_{\ve} c_{v,\ve}B^{1/m_0+1/m_\infty-\ve/(m_0^2m_\infty^2)},
$$
since we have seen that  there are $O_\ve(|v_0v_1|^\ve) $ choices for $a_i,s_i,w_i\in \ZZ$ 
associated to a particular choice of $\v$.
Thus we deduce that 
\begin{equation}\label{eq:trunck}
\begin{split}
M_\v(B)\ll_\ve~&
 \sum_{v_0=a_0w_0}\sum_{v_1=a_0w_1}  \sum_{\substack{|b_0|,|b_1|\leq B^\ve\\
p\mid b_0b_1 \Rightarrow p\in S}
} \sum_{|u_0|,|u_1|\leq B^\ve} L(R_0,R_1)
\\
&\quad +c_{v,\ve} B^{1/m_0+1/m_\infty-\ve/(m_0^2m_\infty^2)}.
\end{split}
\end{equation}

\subsection{Application of the large sieve}

We shall now apply Lemma \ref{large} to estimate 
\eqref{eq:LRR}, which we shall apply with $m=2$.
Let  $\Omega\subset \ZZ^{2}$ be the set of vectors $
\t\in \NN^2$ such that   $\1_{T_0}(t_0)  \1_{T_\infty}(t_1) =1$
and for which  $p\in U_j$ 
whenever there exists an index $j$ such that  $p\| g_j(\t)$.
For any prime $p\not \in S$, 
let 
$$
A_0(p)=\{\t\in (\ZZ/p^2\ZZ)^{2}: \text{$p\mid t_0$ and $p\not\in T_0$}\}
$$
and 
$$
A_\infty(p)=\{\t\in (\ZZ/p^2\ZZ)^{2}: \text{$p\mid t_1$ and $p\not\in T_\infty$}\}.
$$
Similarly, 
 let 
$$
B_j(p)=\{\t\in (\ZZ/p^2\ZZ)^{2}: \text{
$p\| g_j(\t)$ and $p\not\in U_j$}
\}
$$
for $1\leq j\leq r$. 
Then $\#\Omega \bmod{p^2}\leq (1-\overline{\omega}(p))p^4$, where
$$
\overline{\omega}(p)=\frac{\#\left(A_0(p)\cup A_\infty(p)\cup B_1(p)\cup \dots\cup B_r(p)\right)}{p^{4}}. $$
In particular, we have 
$\overline{\omega}(p)\in [0,1)$.
The following result is concerned with estimating this quantity.

\begin{lemma}\label{lem:exp-omega}
Let $p\not \in S$ and let 
$d=\gcd(m_0,m_\infty)$. Then 
$$
\overline{\omega}(p)=
 \frac{\1_{T_0^c}(p)}{p} +
 \frac{\1_{T_\infty^c}(p)}{p} 
 +\sum_{j=1}^r 
\frac{\1_{U_j^c}(p)  \nu_j(p;\v)}{p^{2}}+O\left(\frac{\gcd(p,A_0A_1)}{p^2}\right),
$$
where
$$
\nu_j(p;\v)=\#\{
\t\in \FF_p^{2}: h_j(v_0t_0^d, v_1t_1^d)=0\}.
$$
\end{lemma}

\begin{proof}
Recall that $\gcd(A_0,A_1)=1$, 
that 
$g_j(t_0,t_1)$ is defined in 
\eqref{eq:hj-u}, and that 
$g_j(t_0,t_1)$ is square-free and coprime to the monomial $t_0t_1$. 
If $p\mid A_0A_1$ we take the trivial upper bound 
$$
\#\left(A_0(p)\cup A_\infty(p)\cup B_1(p)\cup \dots\cup B_r(p)\right)=O(p^3),
$$ 
whence
$
\overline{\omega}(p)=O(1/p), 
$
which is satisfactory. 

Suppose henceforth that $p\nmid A_0A_1$.
We proceed by noting that the intersection of any two sets in the 
union $A_0(p)\cup A_\infty(p)\cup B_1(p)\cup \dots\cup B_r(p)$
contains $O(p^{2})$ elements of $(\ZZ/p^2\ZZ)^{2}$.
Thus  
 $$
\overline{\omega}(p)=
 \frac{\1_{T_0^c}(p)}{p} +
 \frac{\1_{T_\infty^c}(p)}{p} 
 +\sum_{j=1}^r 
\frac{\#B_j(p)}{p^{4}}+O\left(\frac{1}{p^2}\right).
$$
Turning to 
$\#B_j(p)
$ for  $j\in \{1,\dots,r\}$, we write $\u=\x+p\y$ for $\x,\y\in \FF_p^{2}$. Thus
$$
\#\{\t\in (\ZZ/p^2\ZZ)^{2}:
p^2\mid g_j(\t)\} =
\sum_{\substack{\x\in \FF_p^{2}\\ g_j(\x)=0}} 
\#\left\{\y\in \FF_p^{2}: \y.\nabla g_j(\x)=-g_j(\x)/p\right\}.
$$
On enlarging $S$, we can assume that 
$\nabla g_j(\x)\neq \0$ for any 
 $\x$ in the sum. 
Thus each of the $O(p)$ values of $\x$ produces $O(p)$ choices of $\y$, giving
$$
\#\{\t \in (\ZZ/p^2\ZZ)^{2}:
p^2\mid g_j(\t)\} =O(p^{2}).
$$
Hence
$$
\#B_j(p)=\1_{U_j^c}(p) p^{2}\#\{\t\in \FF_p^{2}: g_j(\t)=0\} +O(p^{2}).
$$
Putting this  together we have shown that 
$$
\overline{\omega}(p)=
 \frac{\1_{T_0^c}(p)}{p} +
 \frac{\1_{T_\infty^c}(p)}{p} 
 +\sum_{j=1}^r 
\frac{\1_{U_j^c}(p)  \lambda_j(p;A_0,A_1)}{p^{2}}+O\left(\frac{1}{p^2}\right),
$$
where
\begin{align*}
\lambda_j(p;A_0,A_1)
&=\#\{\t\in \FF_p^{2}: 
h_j(A_0t_0^{m_0},A_1t_1^{m_\infty})=0\}.
\end{align*}
for $1\leq j\leq r$.
In order to complete the proof of the lemma, it will suffice to prove that 
\begin{equation}\label{eq:blue_bottle}
\lambda_j(p;A_0,A_1)=\nu_j(p;\v)+O(1),
\end{equation}
for $1\leq j\leq r$, in the notation  of the lemma. 

To see this, let $e$ be the least common multiple of $m_0$ and $m_\infty$, so that $e=m_0m_\infty/d$.
We pick a generator $\alpha\in \FF_p^*$  of $\FF_p^*/(\FF_p^*)^e$.  Then it is easily confirmed that 
$$
\langle \alpha^{de/m_0}\rangle=(\FF_p^*)^d/(\FF_p^*)^{m_0} \quad \text{ and } \quad
\langle \alpha^{de/m_\infty}\rangle=(\FF_p^*)^d/(\FF_p^*)^{m_\infty},
$$
on noting that 
$(\FF_p^*)^{m_0}$ and 
$(\FF_p^*)^{m_\infty}$ are subgroups of 
$(\FF_p^*)^d$. 
(Indeed, to check the first equality, for example,  it suffices to confirm that $\alpha^{de/m_0}$ has order $m_0/d$ 
in $\FF_p^*$.)
The group 
$(\FF_p^*)^d/(\FF_p^*)^{m_0}$ has order $N_0=\gcd(m_0,p-1)$ and, likewise, 
$(\FF_p^*)^d/(\FF_p^*)^{m_\infty}$ has order $N_\infty=\gcd(m_\infty,p-1)$.
It follows from this that any non-zero $d$th
 power in $\FF_p$ can be represented uniquely as
 $u^{m_0} \alpha^{edk/m_0} $ for some $k\in \ZZ/N_0\ZZ$, 
 or as 
  $u^{m_\infty} \alpha^{ed\ell/m_\infty} $ for some $\ell\in \ZZ/N_\infty\ZZ$.
  
 Define 
 $$ 
 \lambda_j(p;A_0,A_1;k,\ell)=
 \#\{
\t\in \FF_p^{2}: h_j(A_0t_0^{m_0}\alpha^{edk/m_0} ,A_1t_1^{m_\infty}\alpha^{ed\ell/m_\infty} )=0\},
$$
for any $k\in \ZZ/N_0\ZZ$ and $\ell\in \ZZ/N_\infty\ZZ$. 
Let $\beta=\alpha^{-edk/m_0-ed\ell/m_\infty}$.
On multiplying through by $\beta^{\deg(h_j)}$
and recalling that  $h_j$ is homogeneous, 
 we
obtain 
 \begin{align*}
 \lambda_j(p;A_0,A_1;k,\ell)
 &=
 \#\{
\t\in \FF_p^{2}: h_j(A_0t_0^{m_0}\alpha^{edk/m_0}\beta ,A_1t_1^{m_\infty}\alpha^{ed\ell/m_\infty}\beta )=0\}\\
&=
 \#\{
\t\in \FF_p^{2}: h_j(A_0t_0^{m_0}\alpha^{-ed\ell /m_\infty} ,A_1t_1^{m_\infty} 
\alpha^{-edk /m_0} )=0\}.
\end{align*}
But  $ed /m_\infty=m_0$ and $ ed /m_0=m_\infty$. Hence a simple change of
variables yields
 \begin{equation}\label{eq:empty_bottle}
 \lambda_j(p;A_0,A_1;k,\ell)
= \lambda_j(p;A_0,A_1;0,0).
\end{equation}

Let $\nu_j^*(p;A_0,A_1)$ denote the contribution to $\nu_j(p;A_0,A_1)$ from $t_0t_1\neq 0$,
and similarly for 
$\lambda_j^*(p;A_0,A_1;k,\ell)$.
Then we may write
\begin{align*}
\nu_j(p;A_0,A_1)
&= \nu_j^*(p;A_0,A_1) +O(1)\\
&=\frac{1}{N_0N_\infty} \sum_{k\in \ZZ/N_0\ZZ} 
\sum_{\ell \in \ZZ/N_\infty\ZZ}   \lambda_j^*(p;A_0,A_1;k,\ell) +O(1)\\
&=  \lambda_j^*(p;A_0,A_1;0,0) +O(1),
\end{align*}
by \eqref{eq:empty_bottle}. Noting that 
$ \lambda_j^*(p;A_0,A_1;0,0)= \lambda_j(p;A_0,A_1) +O(1)$,
we have therefore shown that 
$$
\lambda_j(p;A_0,A_1)=\nu_j(p;A_0,A_1)+O(1).
$$

At this point we recall the factorisation \eqref{eq:factor}, together with the fact that 
$v_i=a_is_iw_i$, for $i=0,1$. 
Hence, since $p\nmid A_0A_1$, a simple change of variables shows that 
\begin{align*}
\nu_j(p;A_0,A_1)
&=
\#\{
\t\in \FF_p^{2}: h_j(v_0 (b_0 s_0^{m_0/d}t_0)^d, v_1(b_1 s_1^{m_\infty/d}t_1)^d)=0\}\\
&=
\nu_j(p;\v),
\end{align*}
from which the claim \eqref{eq:blue_bottle} follows.
\end{proof}

We will need to study the average size of $\overline{\omega}(p)$ as $p$ varies. We break this into the following results.

\begin{lemma}\label{lem:Ti}
We have 
$$
\sum_{\substack{p\leq x\\ p\not \in T_0}} \frac{1}{p}=(1-\delta_{0,\QQ}(\pi))\log\log x+O(1)
$$
and 
$$
\sum_{\substack{p\leq x\\ p\not \in T_\infty}} \frac{1}{p}=(1-\delta_{\infty,\QQ}(\pi))\log\log x+O(1),
$$
in the notation of   \eqref{eq:dD}.
\end{lemma}

\begin{proof}
This is a straightforward consequence of the Chebotarev density theorem, in the form presented by Serre \cite[Thm.~3.4]{serre}, for example.
\end{proof}

Our next result 
concerns the average behaviour of the function 
$\nu_j(p;\v)$ in Lemma 
\ref{lem:exp-omega}, as we average over primes $p\not\in U_j$. This is more difficult and requires 
the use of notation introduced at the start of Section \ref{s:density}, which we recall here.
For a  number field 
$F/\QQ$, let  $\mathcal{P}_F$ denote the set of primes $p\in \ZZ$ that are unramified in $F$ and for which 
there exists a prime ideal $\fp\mid p\fo_F$ of residue degree $1$.
For any positive integer $m\leq [F:\QQ]$ we write 
$\mathcal{P}_{F,m}$ for the subset of $p\in \mathcal{P}_F$ for which there are precisely 
$m$ prime ideals above $p$ of residue degree $1$.

For each  $j\in \{1,\dots, r\}$, 
define the \'etale algebra
$$
N_{E_j,d,v_1/v_0}=\QQ[x]/(r_j(x)),
$$
where 
$
r_j(x)=h_j(x^d,v_1/v_0).
$
As in \eqref{eq:wood}, this has a
factorisation into number fields 
$$
N_{E_j,d,v_1/v_0}=N^{(1)}\times\dots\times N^{(s)},
$$
where $N^{(k)}=N_{E_j,d,v}^{(k)}$,
for $1\leq k\leq s$, where the dependency of $s$ on $j$ is suppressed for legibility.

\begin{lemma}\label{lem:Uj}
For each  $j\in \{1,\dots, r\}$, we have 
$$
\sum_{\substack{p\leq x\\ p\not \in U_j}} \frac{\nu_j(p;\v)}{p^{2}}=
\sum_{k=1}^s
(1-\delta_{D,N^{(k)}}(\pi))\log\log x+
O\left(1+\omega(v_0v_1)\right),
$$
in the notation of   \eqref{eq:dD}, where $\omega(n)$ denotes the number of distinct prime  factors of $n\in \ZZ$.
\end{lemma}

\begin{proof}
We have 
$$
\sum_{\substack{p\leq x\\ p\not \in U_j}} \frac{\nu_j(p;\v)}{p^{2}}
=
\sum_{\substack{p\leq x\\ p\not \in U_j\\ p\nmid v_0v_1}} \frac{\nu_j(p;\v)}{p^{2}}+
\sum_{\substack{p\leq x\\ p\not \in U_j\\ p\mid v_0v_1}} \frac{\nu_j(p;\v)}{p^{2}}
$$
Since $\gcd(v_0,v_1)=1$ the second term is seen to be 
$$
\ll
\sum_{\substack{p\leq x\\ p\mid v_0v_1}} 
\frac{1}{p}\ll 
\omega(v_0v_1).
$$
Next, we see that
$$
\sum_{\substack{p\leq x\\ p\not \in U_j\\ p\nmid v_0v_1}} \frac{\nu_j(p;\v)}{p^{2}}=
\sum_{\substack{p\leq x\\ p\not \in U_j\\ p\nmid v_0v_1}}
 \frac{\#\{t\in \FF_p: h_j(t^d,v_1/v_0)=0\}}{p}+O\left(1\right).
$$
Write $r_j(t)=r_j(t^d,v_1/v_0)$
and let 
$
r_j(t)=r_j^{(1)}(t)\dots r_j^{(s)}(t)
$
be its factorisation into irreducible factors over $\QQ$.
Then $N^{(k)}$ is the number field $\QQ[t]/(r_j^{(k)})$, for $1\leq k\leq s$. 
We have 
$$
\sum_{\substack{p\leq x\\ p\not \in U_j\\ p\nmid v_0v_1}} \frac{\nu_j(p;\v)}{p^{2}}
=
\sum_{k=1}^s\sum_{\substack{p\leq x\\ p\not \in U_j\\ p\nmid v_0v_1}} 
 \frac{\#\{t\in \FF_p: r_j^{(k)}(t)=0\}}{p}+O(1) .
$$
To begin with, it follows from the prime ideal theorem that 
$$
\sum_{\substack{p\leq x}} 
 \frac{\#\{t\in \FF_p: r_j^{(k)}(t)=0\}}{p}=\log \log x +O\left(1+\omega(v_0v_1)\right).
$$
Next, we note that  $p\in U_j$ if and only if $\frob_p$ fixes a component of $S_{E_j}$. 
Let  $F_j$ denote the field of definition of the elements of 
$S_{E_j}$. Then, for any $p\not \in S$, the condition  
$p\in U_j$ is equivalent to the condition   $p\in \mathcal{P}_{F_j}$.
Likewise, for any positive integer $m\leq [N^{(k)}:\QQ]$, 
we will have 
$\#\{t\in \FF_p: r_j^{(k)}(t)=0\}=m$ if and only if
$p\in \mathcal{P}_{N^{(k)},m}$.
Hence 
\begin{align*}
\sum_{\substack{p\leq x\\ p\not \in U_j}} \frac{\nu_j(p;\v)}{p^{2}}
=~&
\sum_{k=1}^s\left(
\log\log x-
\sum_{m=1}^{[N^{(k)}:\QQ]}m
\hspace{-0.2cm}
\sum_{\substack{p\leq x\\
p\in  \mathcal{P}_{N^{(k)},m}\cap \mathcal{P}_{F_j}
}}\frac{1}{p}\right)
+
O\left(1+\omega(v_0v_1)\right).
\end{align*}
The remaining sum over primes is susceptible to a further application of the 
Chebotarev density theorem. Once coupled with 
Theorem \ref{t:densities} and \eqref{eq:new-delta}, this leads to the statement of the lemma. 
\end{proof}

We may combine the previous two results to produce a lower bound for 
the quantity $L(Q)$ in Lemma \ref{large}, with the choice of $\overline{\omega}(p)$ from  Lemma 
\ref{lem:exp-omega}.

\begin{lemma}\label{lem:L(Q)}
For any  $\ve>0$, we  have the lower bound
$$
L(Q)\gg_\ve \frac{(\log Q)^{\Theta_v(\pi)}}{|A_0 A_1|^{\ve}},
$$
where $\Theta_v(\pi)$ is given by \eqref{eq:theta-v}.
\end{lemma}

\begin{proof}
Since $1-\overline{\omega}(p)\leq 1$, we have 
$$
L(Q)\geq \sum_{q\leq Q} \mu^2(q) \prod_{p\mid q} \overline{\omega}(p).
$$
There are many results in the literature concerning mean values of non-negative arithmetic functions. 
However, we can get by with the relatively crude lower bound found in 
\cite[Thm.~A.3]{FI}, which is based on an application of Rankin's trick. Let $\gamma:\NN\to\RR_{\geq 0}$ be a 
multiplicative arithmetic function that is supported on square-free integers and which satisfies
\begin{equation}\label{eq:ab}
\sum_{y<p\leq x}\gamma(p)\log p\leq a\log(x/y)+b,
\end{equation}
for any $x>y>2$, 
for appropriate constants $a,b>0$.  Then it follows from 
\cite[Thm.~A.3]{FI} that
\begin{equation}\label{eq:FI}
\sum_{n\leq x}\gamma(n)\gg \prod_{p\leq x}\left(1+\gamma(p)\right),
\end{equation}
where the implied constant is allowed to depend on $a$ and $b$.
We seek to apply this with
$$
\gamma(n)=\mu^2(n)\prod_{p\mid n}\overline{\omega}(p).
$$

It is clear from Lemma \ref{lem:exp-omega} that 
 $\overline{\omega}(p)=O(1/p)$. Hence
$$
\sum_{y<p\leq x}\gamma(p)\log p
\ll 1+
\sum_{y<p\leq x}
\frac{\log p}{p}\ll 
1+\log(x/y),
$$
uniformly in $v_0$ and $v_1$. Hence \eqref{eq:ab} holds for $a,b=O(1)$ and it follows from \eqref{eq:FI} that 
$$
L(Q)\gg  \prod_{p\leq Q} \left(1+\overline{\omega}(p)\right),
$$
for an absolute implied constant.
On appealing once more to 
Lemma \ref{lem:exp-omega}, 
we find that
 \begin{align*}
\log\left(  \prod_{p\leq Q} \left(1+\overline{\omega}(p)\right)\right)=~&
\sum_{\substack{p\leq Q\\ p\not\in T_0}} \frac{1}{p} 
+
\sum_{\substack{p\leq Q\\ p\not\in T_\infty}} \frac{1}{p} 
+\sum_{j=1}^r 
\sum_{\substack{p\leq Q\\ p\not\in U_j}}
\frac{ \nu_j(p;\v)}{p^{2}}
+O\left(
1+\omega(A_0A_1)
\right).
\end{align*}
These sums are estimated using Lemmas \ref{lem:Ti} and \ref{lem:Uj}, leading to the conclusion that
$$
\log\left(  \prod_{p\leq Q} \left(1+\overline{\omega}(p)\right)\right)=
\widetilde{\Theta}(\pi,v_1/v_0)
\log \log Q
+O\left(
1+\omega(A_0A_1)
\right),
$$
where
$$
\widetilde{\Theta}(\pi,v_1/v_0)=
2-\delta_{0,\QQ}(\pi)
-\delta_{\infty,\QQ}(\pi)  +\sum_{j=1}^r \sum_{k=1}^s\left(1-\delta_{E_j,N_{E_j,d,v_1/v_0}^{(k)}}(\pi)\right),
$$
in the notation of 
  \eqref{eq:dD}. Clearly
  $\widetilde{\Theta}(\pi,v_1/v_0)=\Theta_v(\pi)$, the latter being defined in \eqref{eq:theta-v}.
Hence, the statement of the lemma follows on exponentiating and using the fact that 
$\omega(n)\ll (\log |n|)/(\log\log |n|)
$ for any non-zero $n\in \ZZ$.
\end{proof}

\subsection{Completion of the  proof of  Proposition \ref{prop:2fibres}}

We begin by focusing on the estimation of the quantity 
$L(R_0,R_1)$ that was defined in \eqref{eq:LRR}.
In view of  \eqref{eq:factor}, we see that 
$A_0=v_0 (b_0 s_0^{m_0/d}u_0)^d$ and 
$A_1=v_1 (b_1 s_1^{m_\infty/d}u_1)^d$.
Recall that $s_i\mid v_i$ for $i=0,1$.
Taking $Q=B^\ve$,
we note that 
$$
R_0^{m_0}=\frac{B}{|A_0|}\geq\frac{B}{|v_0 (s_0 b_0u_0)^{m_0}|}\geq B^{1-(1+3m_0)\ve}\geq Q^{4m_0},
$$
provided that $\ve\leq 1/(1+7m_0)$. Similarly, we can assume that 
$R_1\geq Q^4$ if $\ve>0$ is chosen to be sufficiently small. 
Hence, with these choices, we'll have 
$$
(R_0 +Q^{4})(R_1 +Q^{4})
\ll R_0R_1 \ll  \frac{B^{1/m_1+1/m_\infty}}{|A_0|^{1/m_0} |A_1|^{1/m_\infty}}.
$$
We may now apply  Lemma \ref{lem:L(Q)} in Lemma \ref{large} to deduce that 
\begin{align*}
L(R_0,R_1)\ll_\ve ~&
\frac{B^{1/m_0+1/m_\infty}}{|A_0|^{1/m_0}|A_1|^{1/m_\infty}}
\cdot \frac{|A_0A_1|^\ve}{(\log B)^{\Theta_v(\pi)}
  }.
\end{align*}
Substituting this into \eqref{eq:trunck}, recalling \eqref{eq:rad} and summing over $b_0,b_1,u_0,u_1$, the statement of 
Proposition 
\ref{prop:2fibres} easily follows.

\section{Examples: lower bounds and asymptotics}\label{s:examples}

Let $\pi:X\to \PP^1$ be a standard fibration. 
It is clear from the constructions in Section \ref{s:sparsity} that we only be able to interpret local solubility conditions outside a given finite set $S$ of primes.
With more work one might be able to incorporate  local solubility at places in $S$, but this should not change the order of growth, which is the main interest in this paper.
Accordingly, for any finite set $S$ of primes, we introduce the counting function
$$
N_{\text{loc},S}(\pi,B)=\#\left\{
x\in \PP^1(\QQ)\cap \pi(X(\mathbf{A}_{\QQ}^S)): H(x)\leq B
\right\},
$$
where $H$ is the usual height function on $\PP^1(\QQ)$ and $\mathbf{A}_{\QQ}^S$ is the set of ad\`eles away from $S$.
We clearly have $N_{\text{loc},S}(\pi,B)\geq N_{\text{loc}}(\pi,B)$ and we expect these two counting functions to have the same order of magnitude. 

We shall prove several results about Halphen surfaces.
Let $m>1$ be an integer. A {\em Halphen pencil} is a geometrically irreducible pencil 
of plane curves  of degree $3m$ with multiplicity $m$ at $9$ base points $P_1,\dots,P_9$.
We let $X$ be the {\em Halphen surface of order $m$} obtained by blowing up $\mathbb P^2$ at these nine points, as introduced by Halphen \cite{Halphen} in 1882.  
We shall assume that $P_1,\dots,P_9$ are globally defined over $\QQ$, so that $X$ is a smooth, proper, geometrically integral surface defined over $\QQ$. In fact, $X$ is a rational elliptic surface and we obtain a standard morphism $\pi : X \to \mathbb P^1$, such that there exists a unique fibre of multiplicity $m$. In particular, $\pi$  does not admit a section.

\subsection{Lower bounds}

In this section we  establish some  lower bounds for $N_\text{loc,S}(\pi,B)$. 
The following result demonstrates that 
Conjecture \ref{con2} would be false with the exponent $\Delta(\pi)$ and that it is indeed sometimes necessary to take a smaller exponent. 

\begin{thm}\label{t:3}
Let  $\pi:X\to \PP^1$ be a standard fibration. Assume 
it only has non-split fibres above   $0$, $1$ and $\infty$, comprising
geometrically irreducible
double fibres over $0$ and $\infty$, and a non-split fibre of multiplicity one above $1$  that is split by a 
quadratic extension.  Then 
$$
B\ll
N_{\text{loc},S}(\pi,B)\ll B .
$$ 
\end{thm}

\begin{proof}
Suppose that $F=\QQ(\sqrt{d})$ is the quadratic  extension that splits the fibre above $1$, for square-free $d\in \ZZ$. Then it is clear that 
\begin{align*}
0\leq \Theta(\pi)&=\min_{\text{$K/\QQ$ quadratic}}  \left(1-\delta_{1,K}(\pi)\right)
\leq 1-\delta_{1,F}(\pi)=0.
\end{align*}
Hence the upper bound is a direct consequence of Theorem~\ref{t:2}.

For the lower bound we compose the exact counting problem, using Corollary~\ref{cor:onlysplitoverdegree1}. Thus there exists a finite set of places $S$, containing the prime divisors of $2d$, such that
\[
N_{\text{loc},S}(\pi,B)=
 \frac12 \#\left\{ (a,b) \in \ZZp^2\colon 
\begin{array}{l}
|a|,|b|\leq B\\
p \not\in S \Rightarrow 2 \mid v_p(a) \text{ and } 2 \mid v_p(b)\\ 
p \not\in S, p \mid a-b \Rightarrow p \in \mathcal P_F
\end{array}
\right\}.
\]
The lower bound is provided by taking pairs $(a,b)$ of the form $(u^2,d v^2)$.
\end{proof}

In this result we have  $2-\deg \partial_\pi=1$, so that the exponent of $B$ matches the predicted exponent of $B$ in Conjectures \ref{con1} and \ref{con2}. We also have 
$\delta_{0,\QQ}(\pi)=\delta_{\infty,\QQ}(\pi)=1$ and $\delta_{1,\QQ}(\pi)=\frac{1}{2}$, so that  $\Delta(\pi)=\frac{1}{2}$.  However, we saw in the proof that $\Theta(\pi)=0$. 
Thus Theorem \ref{t:3} is in agreement with Conjecture \ref{con2}.

Let us describe what is going on geometrically. Consider the finite \'etale orbifold $\mu_2$-cover $\theta_v \colon \mathbb P^1 \to \mathbb P^1$ given by $(x\colon y) \mapsto (x^2\colon vy^2)$, and the pullback fibrations $\pi_v \colon X_v \to \mathbb P^1$ obtained from normalisation of the pullback of $\pi$ along $\theta_v$.
By Proposition~\ref{prop:thin} we see that the two double fibres of $\pi$ pull back to components of mutiplicity one on $\pi_v$. Also, all fibres which do not lie over $1$ in the composition $X_v \xrightarrow{\pi_v} \mathbb P^1 \xrightarrow{\theta_v} \mathbb P^1$ are split. We proceed by  studying the fibres over $1$.

First we study the fibre of $1$ in $\theta_v$. For $v \in \QQ^\times/\QQ^{\times,2}$, we have $\theta^{-1}(1) = \Spec A$, where $A$ is the degree $2$ \'etale algebra $\QQ(\sqrt{v})$ if $v\not\in \QQ^{\times,2}$, and $\QQ \times \QQ$ for $v \in \QQ^{\times,2}$. This gives
\[
\Delta(\pi_v) = \sum_{D'\mid D} \left(1-\delta_{D'}(\pi_v)\right) = \begin{cases}
0 & \text{ if } v \equiv d,\\
\frac12 + \frac12  = 1 & \text{ if } v \equiv 1,\\
\frac 12 & \text{ otherwise},
\end{cases}
\]
where the sum range over all points $D'$ lying above $D=1 \in (\mathbb P^1)^{(1)}$. In the first case, the fibre over $1$ (which is split by $F$) pulls back to a $F$-point, and becomes split. In the second, case the fibre pulls back to two $\QQ$-points. In the last case, the fibre is irreducible and its residue field is linearly disjoint from the splitting field, and we obtain $\Delta(\pi_v)=\Delta(\pi)$, in general.

Theorem~\ref{t:3} indicates that the main contribution to the point count comes from the single cover $\pi_d$. If we were to exclude the thin set of points coming from this cover, we are left with infinitely many covers $\pi_v$, with $\Delta(\pi_v)= \Delta(\pi)$ for $v\neq 1$.
Proposition \ref{prop:thin}(b) implies that the  covers have no multiple fibres, since it gives
 $$m_{P'} = \frac{m_P}{\gcd\!\left(m_P,e(P'/P)\right)} = \frac{2}{\gcd(2,2)} = 1,
 $$
for each $P'\mid 0, \infty$. Hence, 
in the light of the  original Loughran--Smeets conjecture 
 \cite[Conj.~1.6]{LS},
 we expect the remaining covers to 
 contribute order $B/\sqrt{\log B}$ 
 to the counting function, apart from the cover corresponding to $1$, which should contribute 
 order $B/\log B$. 
 
\medskip

Our second lower bound deals with the case of precisely two non-split fibres and is consistent with 
Conjecture \ref{con2}, 
since $\deg\partial_\pi=2-\frac{1}{m_0}-\frac{1}{m_\infty}$.

\begin{thm}\label{t:ONE}
Let $\pi:X\to \PP^1$ be a standard fibration for which the only non-split fibres lie over $0$ and $\infty$. Then
\[
N_{\text{loc,S}}(\pi,B)\gg 
\frac{B^{\frac{1}{m_0} + \frac{1}{m_\infty}}}{(\log B)^{\Delta(\pi)}}.
\]
\end{thm}

\begin{proof}
We begin by 
using Corollary~\ref{cor:onlysplitoverdegree1} to 
give  explicit conditions for local solubility away from $S$, after passing to a sncd model $X' \to \PP^1$. This leads to the conclusion that $N_{\text{loc},S}(\pi,B)$
is equal to the number of $x=(x_0 \colon x_1) \in \PP^1(\QQ)$ with $H(x) \leq B$, such that 
for each $i\in \{0,1\}$ and every  $p \not \in S$, 
$\frob_p$ fixes a collection of intersecting components
 $Z_j$ of  $X'_{D_i}$ such that $v_p(x_i) \in \langle m(Z_j)\rangle_{\NN}$, 
 where $D_i = V(x_i)$.
The following is clearly a sufficient condition for the fibre over $x$ to have $\mathbb Q_p$-point: for all $i$, the Frobenius $\frob_p$ fixes a component of $Z$ of minimimal multiplicity in $X'_{D_i}$, and $m(Z) \mid v_p(x_i)$.
The density $\partial_i$ of rational primes $p$ for which 
$\frob_p$  fixes an element of  $S_{D_i}$ is equal to $\delta_{D_i}(\pi)=\delta_{D_i,\kappa(D_i)}(\pi)$, in the notation of \eqref{eq:dD}.  Hence the statement of the theorem now follows from Proposition~\ref{prop:maincount} and \eqref{eq:Delta}.
\end{proof}

\subsection{Halphen surfaces with one non-split fibre}

Generically, a Halphen surface has no other non-split fibre apart from the multiple one. Even in these cases the counting problem still depends on the Galois action on the components of the multiple fibres, and how these components intersect. We proceed to  record some results which illustrate this phenomenon, in the course of which it  will be convenient to keep in mind the notation 
\eqref{eq:CS}.


We begin with the following result, which agrees with Conjecture \ref{con2}, since $\deg\partial_\pi=1-\frac{1}{m}$ and $\Delta(\pi)=0$.

\begin{thm}\label{t:TWO}
Let $X \to \mathbb P^1$ be a Halphen surface with a single non-split fibre over $0$, that is the fibre of multiplicity $m$. Suppose that this fibre has a geometric component fixed by $\Gal(\bar{\QQ}/\QQ)$. Then there exists a finite set $S$ such that
\[
N_{\text{loc},S}(\pi,B) \sim c_{\pi,S} B^{1+\frac1m},
\]
where
$$
c_{\pi,S}=
\frac{12c_S(1+\frac{1}{m})}{\pi^2c_S(\frac{1}{m})}
\prod_{p\in  S}\left(1+\frac{1}{p}\right)^{-1}.
$$
\end{thm}

\begin{proof}
By Corollary~\ref{cor:onlysplitoverdegree1} we see that there is a finite set of places $S$ such that
\[
N_{\text{loc},S}(\pi,B) = \frac12 \#\{ (a,b) \in \ZZp^2\colon |a|,|b|\leq B, ~p \not\in S \Rightarrow m \mid v_p(a)\}.
\]
We may apply 
Proposition~\ref{prop:maincount} with $m_0=m$ and $m_1=1$, and with $\mathcal{P}_0=\mathcal{P}_1$ equal to the full set of rational primes. In particular $\partial_0=\partial_1=1$ and it follows that 
$
N_{\text{loc},S}(\pi,B) \sim 
c_{\pi,S}B^{1+\frac{1}{m}}$,
as $B\to \infty$,  where
\begin{align*}
c_{\pi,S}=~&
\frac{2c_S(1+\frac{1}{m})}{c_S(1)c_S(\frac{1}{m})}
\prod_{\substack{p\not \in S}}\left(1-\frac{1}{p^2}\right)
\prod_{p\in  S}\left(1-\frac{1}{p}\right)^2,
\end{align*}
in the notation of  \eqref{eq:CS}.
The statement  easily follows on simplifying the expression for the constant. 
\end{proof}

The following two result agree with Conjecture \ref{con2}, since in both cases we have $\deg\partial_\pi=1-\frac{1}{m}$ and $\Delta(\pi)=\frac{2}{3}$.
Moreover, in these two examples, we have multiple fibres which do not have a geometrically integral component. This demonstrates the need to define \eqref{eq:dD} in terms of $S_D$, for each divisor $D$, which allows us to work with the Galois action on the components of a fibre of minimum multiplicity.

\begin{thm}\label{t:THREE}
Let $X \to \mathbb P^1$ be a Halphen surface with a single non-split fibre over $0$, that is the fibre of multiplicity $m$. Suppose that this fibre consists of three conjugate lines 
split by a cubic Galois extension $K/\QQ$
that do \textbf{not} all meet in a point. Then there exists a finite set $S$ such that
\[
N_{\text{loc},S}(\pi,B) \sim c_{\pi,S} \frac{B^{1+\frac1m}}{(\log B)^{\frac{2}{3}}},
\]
where
\begin{align*}
c_{\pi,S}=~&
\frac{2m^{\frac{2}{3}}
c_S(1)^{\frac{1}{3}}c_S(1+\frac{1}{m})}{\Gamma(\frac{1}{3})c_S(\frac{1}{m})}
\prod_{\substack{p\in \mathcal{P}_K\\ p\not \in S}}
\hspace{-0.1cm}
\left(1+\frac{1}{p}\right)
\left(1-\frac{1}{p}\right)^{\frac{1}{3}}
\hspace{-0.1cm}
\prod_{\substack{p\not \in \mathcal{P}_K\\ p\not \in S}}\left(1-\frac{1}{p}\right)^{\frac{1}{3}}.
\end{align*}
\end{thm}

\begin{proof}
Suppose that the three conjugate lines are split by the cubic Galois extension $K/\QQ$. By Corollary~\ref{cor:onlysplitoverdegree1} we see that there is a finite set of places $S$ such that
$N_{\text{loc},S}(\pi,B) $ is equal to
\[
 \frac12 \#\left\{ (a,b) \in \ZZp^2\colon 
\begin{array}{l}
|a|,|b|\leq B\\
\big[
p \not\in S \text{ and } p \mid a\big] \Rightarrow \big[m \mid v_p(a) \text{ and } p \in \mathcal P_K \big]
\end{array}
\right\},
\]
where $\mathcal{P}_K$ is the set of rational primes $p$ that are unramified in $K$ and split completely. 
We may apply 
Proposition~\ref{prop:maincount} with $m_0=m$ and $m_1=1$, and with $\mathcal{P}_0=\mathcal{P}_K$ and $\mathcal{P}_1$  equal to the full set of rational primes. In particular $\partial_0=1/3$ and $\partial_1=1$.  
It follows that 
$$
N_{\text{loc},S}(\pi,B) \sim c_{\pi,S} \frac{B^{1+\frac1m}}{(\log B)^{\frac{2}{3}}}
$$
where
\begin{align*}
c_{\pi,S}=~&
\frac{2m^{\frac{2}{3}}}
{\Gamma(\frac{1}{3})}
\cdot 
\frac{c_S(1+\frac{1}{m})}{c_S(\frac{1}{m})}
\prod_{\substack{p\in \mathcal{P}_K\\ p\not \in S}}\left(1-\frac{1}{p^2}\right)
\\
&\quad \times 
\prod_{p\in  \mathcal{P}_K\cap S}\left(1-\frac{1}{p}\right)
\prod_{p\in \mathcal{P}_K}\left(1-\frac{1}{p}\right)^{-\frac{2}{3}}
\prod_{p\not \in \mathcal{P}_K}\left(1-\frac{1}{p}\right)^{\frac{1}{3}}.
\end{align*}
The statement of the proposition  follows on simplifying this expression.
\end{proof}

The next result agrees with Conjecture \ref{con2}, since $\deg\partial_\pi=1-\frac{1}{m}$ and $\Delta(\pi)=\frac{2}{3}$.

\begin{thm}\label{t:FOUR}
Let $X \to \mathbb P^1$ be a Halphen surface with a single non-split fibre over $0$, that is the fibre of multiplicity $m$. Suppose that this fibre consists of three conjugate lines
split by a cubic Galois extension $K/\QQ$
 that \textbf{do} meet in a point. Then there exists a finite set $S$ such that
\[
\frac{B^{1+\frac1m}}{(\log B)^{\frac{2}{3}}}\ll 
N_{\text{loc},S}(\pi,B) \ll  \frac{B^{1+\frac1m}}{(\log B)^{\frac{2}{3}}}.
\]
\end{thm}

\begin{proof}
The upper bound follows from Theorem \ref{t:1}.
The lower bound was proven in Theorem~\ref{t:ONE}.
\end{proof}

Theorem \ref{t:FOUR} illustrates the need for the non-split fibres to be sncd; the counting problem for this setting is
\[
p \not\in S, p \mid a \Rightarrow \Big[\big(3m \mid v_p(a)\big) \text{ or } \big(m \mid v_p(a) \text{ and } p \in \mathcal P_K\big)\Big].
\]
The condition $3m \mid v_p(a)$ comes from a Galois fixed component of multiplicity $3m$ on the multiple fibre of the sncd-model of $X$. However, no such component exists on the multiple fibre of $X$ itself.

\subsection{Halphen surfaces with two non-split fibres}

In practice, it can be difficult to  construct Halphen surfaces with more than one non-split fibre. We present two such examples, both of which  verify Conjecture \ref{con2}.

\begin{thm}\label{t:SIX}
There exists a Halphen surface $X \to \mathbb P^1$ of degree $2$ with two non-split fibres: the multiple fibre is geometrically irreducible and has multiplicity $2$, and the other is a sncd divisor of Kodaira classification $I_6$ split by a cubic Galois extension $K/\QQ$. Moreover, there exists a finite set of places $S$, and an explicit constant $c_{\pi,S}>0$ such that
\[
N_{\text{loc},S}(\pi,B) \sim  c_{\pi,S}\frac{B^{1+\frac12}}{(\log B)^{\frac{2}{3}}}.
\]
\end{thm}

\begin{proof}
Let us first fix the cyclic cubic number field $K/\QQ$. Now choose two sets of three conjugate points $P_i ,Q_i \in \mathbb P^2(K)$, indexed by $i \in \ZZ/3\ZZ$. We let $R_i$ be the intersecting point of the lines $P_{i+1}P_{i+2}$ and $Q_{i+1}Q_{i+2}$. For generic choices of $P_i$ and $Q_i$, the $R_i$ are well-defined and there is a unique smooth cubic through the nine points $P_i$, $Q_i$ and $R_i$.

We will consider $X = \Bl_{P_i,Q_i,R_i} \PP^2$. The two non-split fibres of $X$ come from the double cubic passing through these nine points, and the sextic curve which is geometrically the union of the six lines $P_{i+1}P_{i+2}$ and $Q_{i+1}Q_{i+2}$. Under blowup the first curve turns into a geometrically integral fibre of multiplicity $2$, and the other into six lines meeting in a cycle. The three lines $P_1P_2$, $P_2P_3$ and $P_3P_1$ are permuted by $\Gal(K/\QQ)$ and no longer meet on $X$.
For a generic choice of $P_i$ and $Q_i$ there will be no other non-split fibres.

Let us assume the multiple fibre lies above $0$ and the other non-split fibre over $\infty$. The fibres of $X \to \mathbb P^1$ are all sncd, so we can directly compose the counting problem to find that
\[
N_{\text{loc},S}(\pi,B)=
 \frac12 \#\left\{ (a,b) \in \ZZp^2\colon 
\begin{array}{l}
|a|,|b|\leq B\\
p \not\in S \Rightarrow 2 \mid v_p(a)\\
\big[ 
p \not\in S \text{ and } p \mid b\big] \Rightarrow p \in \mathcal P_K
\end{array}
\right\}.
\]
Such a counting problem is dealt with by Proposition \ref{prop:maincount}.
\end{proof}

\begin{thm}\label{t:SEVEN}
There exists a Halphen surface $X \to \mathbb P^1$ of degree $3$ with two non-split fibres: the multiple fibre is geometrically irreducible and has multiplicity $3$, and the other is a non-sncd divisor of Kodaira classification $I_3$ split by a cubic Galois extension $K/\QQ$. Moreover, there exists a finite set of places $S$ such that
\[
\frac{B^{1+\frac13}}{(\log B)^{\frac{2}{3}}} \ll N_{\text{loc},S}(\pi,B) \ll  \frac{B^{1+\frac13}}{(\log B)^{\frac{2}{3}}}.
\]
\end{thm}

We will return to this surface in Section \ref{s:7.4} to create another interesting example. There we will assume that the multiple fibre lies over $0$ and the remaining non-split fibre lies over $\infty$.

\begin{proof}[Proof of Theorem \ref{t:SEVEN}]
Let 
$E/\QQ$ be an elliptic curve with $E(\QQ)_{\text{tors}} =\ZZ/9\ZZ$. Let 
$K/\QQ$ be a cyclic cubic number field $K/\QQ$, such that
$\rank E(\QQ) < \rank E(K)$.
We will fix
\begin{itemize}
\item[(i)] a generator $\sigma \in \Gal(K/\QQ)$,
\item[(ii)] a generator $A \in E(\QQ)_{\text{tors}}$,
\item[(iii)] $B \in E(K)\setminus E(\QQ)$ such that $B+\sigma(B) + \sigma^2(B) = O \in E(\QQ)$, and any
\item[(iv)] $C \in E(K) \setminus E(\QQ)$.
\end{itemize}

With this notation in mind,  consider the nine points
\[
P_i = \sigma^i(C), \quad Q_i = \sigma^i(-2C + B +A) \quad \text{ and } R_i = \sigma^i(C+2A).
\]
For general choices, we find that 
$\Bl_{P_i,Q_i,R_i} \mathbb P^2$ is a Halphen surface of degree $3$. In particular,
there is a smooth cubic through the nine points, which becomes the geometrically irreducible triple fibre on $X$. Moreover, we have 
$$
\sum_i (P_i+Q_i+R_i) - P_j + R_j = O,
$$ 
so that  there is a cubic curve which passes through all nine points except $P_j$ and has a singularity at $R_j$. The union of these three curves becomes the the non-split $I_3$-fibre, split by $K$.

For the lower bound we may apply Theorem \ref{t:ONE} and the upper bound follows from  Theorem \ref{t:1}.
\end{proof}

\subsection{A non-split fibre over a point of higher degree}\label{s:7.4}

Our final result concerns  a surface of Halphen type, with a fibration over $\mathbb P^1$ that has one multiple fibre and a non-split fibre over a degree $2$ point. Our local solubility criteria do not apply to  this case in general, but we are nonetheless able to  deduce explicit criteria.

Consider the Halphen surface $\pi:X \to \PP^1$ from Theorem
\ref{t:SEVEN} 
with $m=3$, 
with a multiple fibre over $0$ and a non-split fibre over $\infty$ split by a Galois cubic extension $K/\QQ$. Let $\pi':X' \to \PP^1$ be the normalisation of the pullback of $\pi$ along the morphism $\theta \colon \mathbb P^1 \to \PP^1$ given by $[u \colon v] \mapsto [u^2 \colon u^2+v^2]$.
We claim that the surface $X'$ has a unique multiple fibre over $u=0$, whose multiplicity is $3$, and that the only other non-split fibre lies over the degree $2$ point $u^2+v^2=0$, and is split by $K$. To see this we note that the fibres of the pullback of $X$ are precisely the fibres of $X'$, and   normalisation only changes the fibres over $0$ and $\infty$. The multiplicities of the new fibres can then be computed using Proposition~\ref{prop:multiplicities under normalisation}.
Note that 
$\partial_{\pi'} =\frac{2}{3}[0]$ and 
$
\Delta(\pi')=1-\delta_{u^2+v^2}(\pi')=\frac{2}{3}.
$ 
We shall now prove the following result, which is easily seen to agree with the prediction in Conjecture \ref{con2}.

\begin{thm}\label{t:FIVE}
For the surface $\pi':X'\to \PP^1$ as above,
there exists a finite set $S$ such that
\[
\frac{B^{\frac43}}{(\log B)^{\frac23}} \ll N_{\text{loc},S}(\pi',B) \ll  \frac{B^{\frac43}}{(\log B)^{\frac23}}.
\]
\end{thm}

\begin{proof}
The upper bound follows directly from
Theorem \ref{t:1}. To prove the lower bound, we note that 
for all  but finitely many points $x \in \PP^1(\QQ)$, the fibre of $X' \to \PP^1$ is isomorphic to the fibre of $X \to \PP^1$ over $\theta(x) \in \PP^1(\QQ)$. Hence we can apply
the criterion in Corollary~\ref{cor:onlysplitoverdegree1} to determine  local solubility for $X$.
Noting that $v_p(u^2)$ is divisible by $3$ precisely if this is true for $v_p(u)$, we find that $N_{\text{loc},S}(\pi',B)$ is 
\begin{align*}
\frac12 \#
\left\{  (u,v)\in \ZZp^2: 
\begin{array}{l}
|u|,|v|\leq B\\
p\not\in S \Rightarrow 3\mid v_p(u)\\
\big[p\not\in S \text{ and }
p\mid u^2+v^2\big] \Rightarrow p\in \mathcal{P}_K 
\end{array}
\right\} + O(1).
\end{align*}
On restricting to 
positive coprime $u,v$ 
and demanding that $u$ is cube, we arrive at the lower bound
$$N_{\text{loc},S}(\pi',B)\geq \frac{1}{2}M(B)+O(1),
$$ 
where
\begin{align*}
M(B)=\#
\left\{  (u,v)\in \ZZp^2: 
\begin{array}{l}
0\leq u^3,v\leq B\\
p\mid u^6+v^2  \Rightarrow p\in \mathcal{P}_K
\end{array}
\right\}.
\end{align*}
Note that $u^3,v\leq B$ whenever $u^6+v^2\leq B^2$. Hence
$$
M(B)\geq \#
\left\{  (u,v)\in \ZZ_{\geq 0}^2: 
\begin{array}{l}
\gcd(u,v)=1, ~u^6+v^2\leq B^2\\
p\mid u^6+v^2  \Rightarrow p\in \mathcal{P}_K
\end{array}
\right\}.
$$
The right hand side is exactly the quantity estimated via the $\beta$-sieve
 by Friedlander and Iwaniec \cite[Thm.~11.31]{FI}, with the outcome that 
$$
M(B)\gg \left(\frac{B^2}{\log(B^2)}\right)^{\frac{2}{3}}.
$$
The statement of the theorem now follows.
\end{proof}

\end{document}